\documentclass[11pt, a4paper, oneside]{amsart}
\usepackage{latexsym,amssymb,amsmath} 
\usepackage{cite} 
\usepackage{path}
\usepackage{url}
\usepackage{verbatim}
\usepackage[pdftex]{graphicx}
\usepackage[utf8]{inputenc}
\usepackage[T1]{fontenc}
\usepackage{multirow}
\usepackage{float}
\usepackage{soul}

\usepackage{color}
\usepackage{hyperref}

\newtheorem{remark}{Remark}
\newtheorem{algorithm}{Algorithm}
\newtheorem{theorem}{Theorem}
\newtheorem{definition}{Definition}

\newtheorem{lemma}{Lemma}

\numberwithin{equation}{section}
\numberwithin{table}{section}
\numberwithin{figure}{section}

\newcommand{\iS}{\mathcal{S}}
\newcommand{\iA}{\mathcal{A}}
\newcommand{\iI}{\mathcal{I}}
\newcommand{\iJ}{\mathcal{J}}
\newcommand{\iL}{\mathcal{L}}
\newcommand{\iK}{\mathcal{K}}
\newcommand{\iV}{V_{\text{ad}}}

\newcommand{\iG}{\mathcal{G}}
\def\lag{\langle}
\def\rag{\rangle}

\newcommand{\bu}{\mathbf{u}}
\newcommand{\bp}{\mathbf{p}}
\newcommand{\bl}{\boldsymbol{\lambda}}
\newcommand{\IN}{\mathbb{N}}

\newcommand{\RR}{\mathbb{R}}

\newcommand{\IR}{\mathbb{R}}

\newcommand{\Div}{\mathrm{Div}}

\newcommand{\II}{\mathbf{I}}

\newcommand{\yy}{\mathbf{y}}

\newcommand{\junk}[1]{{}}
\newlength{\fwtwo} 

\begin{document}
\DeclareGraphicsExtensions{.jpg}
\renewcommand*{\thefootnote}{\fnsymbol{footnote}}
\renewcommand*{\thefootnote}{$\dag$}

\author{C. Chung$^\dag$, J.C. De los Reyes$^\dag$ and C.B. Sch\"onlieb$^\ddag$}
\address{$^\dag$Research Center on Mathematical Modelling (MODEMAT), EPN Quito, Quito, Ecuador}
\email{cao.vanchung@epn.edu.ec}
\email{juan.delosreyes@epn.edu.ec}
\address{$^\ddag$Department of Applied Mathematics and Theoretical Physics, University of Cambridge, Cambridge, United Kingdom.}
\thanks{$^*$This research has been supported by SENESCYT through Prometeo program and MATH-AmSud project SOCDE ``Sparse Optimal Control of Differential Equations''. CBS acknowledges support from the {EPSRC} grant Nr.~EP/M00483X/1 and from the Leverhulme grant `Breaking the non-convexity barrier'}
\email{cbs31@cam.ac.uk}

\title[Optimal spatially-dependent regularization]{Learning optimal spatially-dependent regularization parameters in total variation image denoising$^*$}



\maketitle

\begin{abstract}
We consider a bilevel optimization approach in function space for the choice of spatially dependent regularization parameters in TV image denoising models. First- and second-order optimality conditions for the bilevel problem are studied when the spatially-dependent parameter belongs to the Sobolev space $H^1(\Omega)$. A combined Schwarz domain decomposition-semismooth Newton method is proposed for the solution of the full optimality system and local superlinear convergence of the semismooth Newton method is verified. Exhaustive numerical computations are finally carried out to show the suitability of the approach.
\end{abstract}

\section{Introduction}

The idea of Total Variation (TV) regularization for removing the noise in a given noisy image $f$ consists in reconstructing a denoised version $u$ of it by minimizing the generic functional
\begin{equation}\label{TVfunc}
\mathcal F (u) = |Du|(\Omega)+\int_\Omega\lambda \, \phi(u, f) \, dx
\end{equation}
where
$$
|Du|(\Omega) = \underset{v\in C^\infty_0(\Omega, \IR^2),\|v\|\leq 1}{\sup}\int_\Omega u \nabla \cdot v \, dx
$$
is the total variation (TV) of $u$ in $\Omega$, $\lambda$ is a positive parameter function and $\phi$ is a suitable fidelity function, dependent on the type of noise included in $f$. The parameter $\lambda$ can be either a positive constant or a spatially dependent function $\lambda:\Omega\rightarrow \RR^+$. If $\lambda \in \RR^+$, the parameter serves as a homogeneous weight between the fidelity measure and the TV-regularizing term. On the other hand, if $\lambda$ is considered as spatially dependent, i.e., $\lambda:\Omega\rightarrow \RR^+$, it can also reflect information on possibly heterogenous noise in the image, as well as making a difference between regularization of small and large scale features in the image. Hence, $\lambda$ has a key role in spatially balancing the amount of regularization. Spatially dependent parameters have been considered in the recent papers \cite{bredies2013spatially,dong2011automated,KFrick_Marnizt2012,lauzier2012non,lauzier2012noise}.

The choice of an appropriate regularization parameter $\lambda$ is a difficult task and has been the subject of many research efforts (see, e.g., \cite{dong2011automated, KFrick_Marnizt2012, KFrick_Marnizt2012_2, KFrick_Marnizt2013, Gilboa_et_el_2006, DStrong_et_el_2006,Vogel_C_R, tadmor2004multiscale}). In \cite{Juan_Carlos_Carolina}, a bilevel optimization approach in function space was proposed for learning the weights in \eqref{TVfunc}. In the flavour of supervised machine learning, the approach presupposes the existence of a training set of clean and noisy images. Existence of Lagrange multipliers was proved and an optimality system characterizing the solution was obtained. The analytical results hold both for $\lambda\in \RR^+$ and $\lambda:\Omega\rightarrow\RR^+$, while a solution algorithm was only designed for solving the bilevel optimization problem with $\lambda\in \RR^+$. A related approach for finite-dimensional variational problems was proposed in \cite{kunischpock}.

In Figure \ref{fig:motiv1} the influence of the choice of a constant $\lambda$ in \eqref{TVfunc} is shown, over-regularising the reconstructed image if chosen too small and under-regularising if chosen too large. Moreover, in Figure \ref{fig:motiv2} the reconstructed images with constant and spatially-dependent $\lambda$ are shown, where $\lambda$ has been optimized with the bilevel approach for \eqref{TVfunc} proposed in \cite{Juan_Carlos_Carolina}.

\begin{figure}
\includegraphics[width=5.5cm]{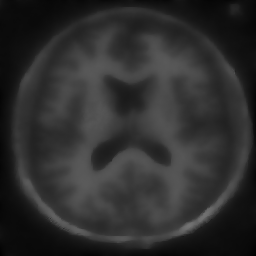} \hspace{1.5cm} \includegraphics[width=5.5cm]{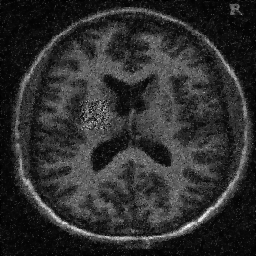}
\caption{TV denoised images that have been computed as minimizers of \eqref{TVfunc} with different choices for $\lambda\in \RR^+$. While choosing $\lambda$ too small is over-regularizing the image, choosing it too large is under-regularizing, the question is what the best choice of $\lambda$ is and how to compute it.}
\label{fig:motiv1}
\end{figure}

\begin{figure}
\includegraphics[width=5.5cm]{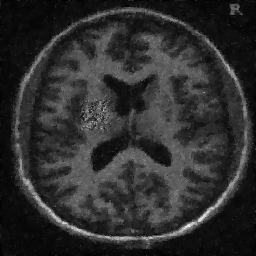} \hspace{1.5cm} \includegraphics[width=5.5cm]{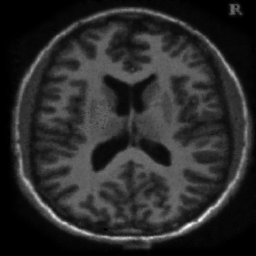}
\caption{TV denoised images, computed as minimizers of \eqref{TVfunc}, with a $\lambda$ that has been optimised with the bilevel approach in \cite{Juan_Carlos_Carolina}. On the left the optimally computed $\lambda\in \RR^+$ is constant, on the right the $\lambda$ is spatially dependent, computed with the approach proposed and analysed in this paper.}
\label{fig:motiv2}
\end{figure}

In this article we consider the bilevel optimization approach for \eqref{TVfunc} from \cite{Juan_Carlos_Carolina}, with a spatially dependent parameter $\lambda \in H^1(\Omega)$ and $\phi(\cdot)=(\cdot)^2$  as presented in Section \ref{secimport}, and investigate first- and second-order optimality conditions for the bilevel problem. In addition to the nonsmooth lower level denoising problems, a positivity constraint on the functional parameter ($\lambda \geq 0$ a.e. in $\Omega$) has to be imposed to guarantee well-posedness. These elements lead to a nonlinear and nonsmooth first-order optimality system with complementarity relations.

For proving second order sufficient optimality conditions, we improve previous G\^ateaux differentiability results of the solution mapping\cite{Juan_Carlos_Carolina} and show that it is actually twice Fr\'echet differentiable under suitable assumptions. We then define a cone of critical directions and prove the result by utilizing a contradiction argument.

Since the resulting optimality system involves several coupled PDEs (twice the size of the training set), the efficient numerical solution of the problem becomes challenging. We consider a combined Schwarz domain decomposition-semismooth Newton approach, where the domain $\Omega$ is subdivided into overlapping subdomains $\Omega_i$ with optimized transmission conditions (see, e.g., \cite{Quarteroni_Valli, Grande2006, E_Okyere}). We apply Schwarz domain decomposition methods directly to the nonlinear optimality system rather than to a linearization of it, and solve, in each subdomain, a reduced nonlinear and nonsmooth optimality system. We propose a semismooth Newton algorithm for the solution of each subdomain system and analyze the local superlinear convergence of the method.

The outline of the paper is as follows. In Section 2 the bilevel optimization problem is stated and analyzed. The analysis involves differentiability properties of the solution operator and the derivation of first and second order optimality conditions. The numerical treatment of the problem is considered in Section 3. The discretization of the problem is described and the domain decomposition and semismooth Newton algorithms are presented. Also the convergence analysis of the semismooth Newton method is carried out. Finally, in Section 4 an exhaustive numerical experimentation is presented. We compare our approach with other spatially-dependent approaches and apply it to problems with large training sets.

\section{The bilevel optimization problem in function space} \label{secimport}
Bilevel optimization encompasses a general class of constrained optimization problems in which the constraint constitutes an optimization problem itself, which is called the lower level problem. The idea of employing bilevel optimization for learning variational image processing approaches arises as minimizing a quality measure for the solution of the variational approach with respect to free parameters in the model. That is, we consider the problem
\begin{align*}
& \min_\lambda ~ C(u(\lambda))\\
& \textrm{s.t. } u(\lambda) \in \mathrm{argmin}_u \mathcal J(u,\lambda),
\end{align*}
where $\lambda$ encodes the free parameters and $C$ is a quality measure for a minimizer of the functional $\mathcal J$. If $\mathcal J$ is the TV denoising functional \eqref{TVfunc} such a free parameter is the regularization parameter $\lambda$. The most standard quality measure used in the bilevel context is the mean of $L^2$ squared distances of solutions of the variational model to desirable examples that are given in form of a training set. For learning variational image denoising models such a training set consists of noisy images and the corresponding clean/true images. In other contexts the training set will be different, e.g. for image segmentation the training set might consist of the to be segmented image and the true segmentation. Once the parameters in the variational model are learned on the basis of the training set, then the learned model is used for new image data. See \cite{calatroni2015bilevel} for a recent review on bilevel learning in image processing.

In the context of learning image processing approaches, the constraint problem is typically non-smooth --- as with TV regularization as in \eqref{TVfunc} --- making its robust numerical solution a challenging topic. In particular, the derivation of sharp, analytic optimality conditions usually requires twice-continuous differentiability of the functional in the lower level problem and invertibility of its Hessian. Roughly, this is because the solution of the lower level problem does in general not have an explicit expression and we therefore have to apply the implicit function theorem for being able to insert it in the optimality condition for the upper level problem. A successful strategy for dealing with non-smooth lower level problems, therefore, are targeted, active-inactive set smoothing approaches, such as smoothing the TV with Huber regularization \cite{Juan_Carlos,Juan_Carlos_Carolina,kunischpock}. Another recent proposal for the computational realization of bilevel problems with non-smooth constraints can be found in \cite{ochs2015bilevel}, where the lower level problem is approximated by an iteration of sufficiently smooth update rules. The latter has been derived considering the discrete bilevel problem. In contrast, deriving the optimality conditions for the smoothed-problem in function space as in \cite{Juan_Carlos,Juan_Carlos_Carolina}, following the principle of optimize-then-discretize rather than discretize-then-optimize, has the advantage that these conditions can be used to construct resolution independent iterative schemes \cite{hintermuller2006infeasible}. This is the approach that we too pursue in this paper.

We consider the bilevel problem for learning the parameter $\lambda$ for a smoothed version of the TV denoising model in \eqref{TVfunc}. Given a training set $(u_i^\dag,f_i), ~i=1, \dots, N,$ of true and noisy images, respectively, the bilevel optimization problem under consideration reads as follows:
Find a minimizer $(u^*_1, \dots,u^*_N, \lambda^*)\in [H^1_0(\Omega)]^N \times H^1(\Omega)$ of the problem
\begin{subequations}\label{eq00_1}
\begin{align}
& \underset{(u_1, \dots, u_N,\lambda) \in [H^1_0(\Omega)]^N \times H^1(\Omega)}{\min }J(u_1, \dots,u_N, \lambda) := \sum_{i=1}^N \|u_i-u_i^\dag\|^2_{L^2} + \beta\|\lambda\|^2_{H^1(\Omega)} \label{eq00_1a}\\
&\text{subject to:} \notag\\
& \langle e_i(u_i,\lambda), v\rangle_{ H^{-1},  H^1_0} = \mu \big(D u_i, Dv\big)_{L^2} + \big(h_\gamma(Du_i), Dv\big)_{L^2}\label{eq00_1b}\\ & \hspace{4cm} + \int\limits_\Omega{\lambda \phi'(u_i, f_i)v dx} = 0\quad
\text{ for all } v\in H^1_0(\Omega), ~i=1, \dots, N, \notag\\
& \lambda \geq 0 \quad \text{ a.e. in }\Omega,
\end{align}
\end{subequations}
where $N$ is the size of the training set of images, $0 < \mu \ll 1$, $e_i: H^1_0(\Omega) \times H^1(\Omega) \rightarrow  H^{-1}(\Omega)$, for $i=1, \dots, N$, and
$$\phi(u_i, f_i)=(u_i-f_i)^2, ~i=1, \dots, N.$$
Equations \eqref{eq00_1b} correspond to the necessary and sufficient optimality conditions of a regularized version of the total variation denoising models. In this manner, we replace the lower level minimization problems by an equivalent system of partial differential equations.

The $C^2$-regularizing function $h_\gamma: \RR^n \to \RR^n$ is given by:
\setcounter{equation}{1}
\begin{equation}\label{eq00_2}
h_\gamma(z) =
\begin{cases}\frac{z}{|z|}\qquad & \mbox{if}\quad \gamma|z| \geq b,\\
\begin{aligned}
\frac{z}{|z|}\bigg\{\frac{2\gamma - 1}{4\gamma} &+ \frac{\gamma|z|}{2} - \frac{\gamma}{2}\big(\gamma|z|-a\big)\big(\gamma|z| - b\big)\\
 &+ \frac{\gamma^3}{2}\big(\gamma|z|-a\big)^2\big(\gamma|z|-b\big)^2 \bigg\}
\end{aligned} \qquad & \mbox{if}\quad a < \gamma|z| \leq b,\\
\gamma z \qquad & \mbox{else,}
\end{cases}\end{equation}
where $a := 1-\frac{1}{2\gamma}$, $b := 1+\frac{1}{2\gamma}$, $|\cdot|$ stands for the euclidean norm and the division has to be understood componentwise. This function locally regularizes the subgradient of the TV-norm around $0$. Note that the smoothing applied to the TV denoising problem firstly smoothes the TV with \eqref{eq00_2}, and secondly adds a small elliptic regularization term (weighted by $\mu$) to the functional which results in the weak optimality condition in \eqref{eq00_1b}. We have outlined the reason for the Huber regularization above. The reason for the addition of the elliptic term $\mu\|Du\|_2^2$ to \eqref{TVfunc} is, that it numerically renders the inversion of the Hessian of the lower level functional more robust and that it places the problem in Hilbert space and therefore opens up a large toolbox for the analysis of the smoothed problem and its approximation properties, see also \cite{de2016structure}.

The next result involves some properties of $h_\gamma$, which will be used throughout the paper.
\begin{lemma}\label{characters_h_gamma}
The first and second derivative of the function $h_\gamma: \RR^n \to \RR^n$ are Lipschitz continuous functions, with Lipschitz constants depending only on $\gamma$.
\end{lemma}
\begin{proof}
The proof is included in the appendix (Section \ref{sec:Appendix}).
\end{proof}

In order to simplify the presentation, we focus hereafter on the case $N=1$. The results are, however, easily extendable to larger training sets, as will be shown in Section 4.

\subsection{Differentiability of the solution operator}\label{sub:Differentiability and optimality system}

From \cite{Juan_Carlos_Carolina} we know that for each fixed $\gamma > 0$, there exists an optimal solution for problem (\ref{eq00_1}). Denoting by $G:  H^1(\Omega) \rightarrow  H^1_0(\Omega)$ the solution operator $
G(\lambda) =  u$, where $u$ is solution of equation \eqref{eq00_1b} corresponding to $\lambda\in  H^1(\Omega)$, it has been shown in \cite{Juan_Carlos_Carolina} that the operator is G\^ateaux differentiable. In the next theorem we improve that result and prove that the solution operator is actually Fr\'echet differentiable.

\begin{theorem}\label{frechet_diff_G}
Let $f \in L^p(\Omega),$ for some $p>2$, and $\lambda \in \iV:=\{v\in H^1(\Omega): v\geq 0 \text{ a.e. in } \Omega\}$. Let further $B(\lambda)$ be a neighbourhood of $\lambda$. Then, the solution operator
\begin{align*}
G: B(\lambda) & \rightarrow H^1_0(\Omega)\\
 \tilde \lambda & \mapsto u(\tilde \lambda),
\end{align*}
where $u(\tilde \lambda)$ is the solution to $(\ref{eq00_1b})$ associated to $\tilde \lambda$, is Fr\'echet differentiable on $B(\lambda)$ and its derivative at $\lambda\in\iV$, in direction $\xi\in H^1(\Omega)$, is given by $z^\xi_\lambda = G'(\lambda)\xi \in H_0^1(\Omega)$, which corresponds to the unique solution of the linearized equation:
\begin{equation}\label{equ_g_gamma_d}\mu(D z^\xi_\lambda, Dv\big)_{L^2} + \big(h'_\gamma(Du) Dz^\xi_\lambda, Dv\big)_{L^2} + 2\int_\Omega{\lambda z^\xi_\lambda v} \, dx+ 2\int_\Omega{\xi (u - f) v} \, dx= 0, \forall v \in H^1_0(\Omega).
\end{equation}
\end{theorem}
\begin{proof}
Along this proof we denote by $C$ a generic positive constant which may depend on $\gamma$ and $\lambda$. Let us also denote by $u$ and $u_\xi$ the corresponding solutions to (\ref{eq00_1b}) with $\lambda$ and $\lambda + \xi$, respectively. By monotonicity techniques (see \cite[Thm.~2.7]{de2015numerical}), we obtain the existence of a unique solution $u_\xi$, for $\|\xi\|_{H^1(\Omega)}$ sufficiently small, and a unique solution $z^\xi_\lambda \in H_0^1(\Omega)$ to (\ref{equ_g_gamma_d}). Moreover, we get the estimates
\begin{equation} \label{eq: a priori estimates of state and lin. eq}
\|u_\xi - u\|_{H^1_0} = O(\|\xi\|_{H^1}),\quad \|z^\xi_\lambda\|_{H^1_0} = O(\|\xi\|_{H^1}).
\end{equation}

By taking the difference between (\ref{eq00_1b}), with $\lambda$ and $\lambda + \xi$, and (\ref{equ_g_gamma_d}) we get that
\begin{align*}\mu\big(D(u_\xi - u - z^\xi_\lambda), Dv\big)_{L^2} &+ \big( h_\gamma(Du_\xi) - h_\gamma(Du) - h'_\gamma(Du)Dz^\xi_\lambda, Dv\big)_{L^2} \\
&+ 2\int_\Omega\lambda(u_\xi - u - z^\xi_\lambda)v \, dx+ 2\int_\Omega\xi(u_\xi - u)v \, dx= 0, \quad \forall v\in H^1_0(\Omega).
\end{align*}
Introducing $\eta := u_\xi - u - z^\xi_\lambda$, we can write the last equation as follows
\begin{align*}\mu\big(D\eta, Dv\big)_{L^2} &+ \big(h'_\gamma(Du)D\eta, Dv\big)_{L^2} + 2\int_\Omega\lambda\eta v \, dx= - 2\int_\Omega\xi(u_\xi - u)v \, dx\\
&- \big( h_\gamma(Du_\xi) - h_\gamma(Du) - h'_\gamma(Du)D(u_\xi - u), Dv\big)_{L^2},\quad\forall v\in H^1_0(\Omega).
\end{align*}
Taking $v = \eta$ and using the monotonicity of $h'_\gamma(Du)$ and $\lambda \geq 0$ a.e. in $\Omega$, we get that
$$\|\eta\|^2_{H^1_0} \leq \bigg| \big(h_\gamma(Du_\xi) - h_\gamma(Du) - h'_\gamma(Du)D(u_\xi - u), D\eta\big)_{L^2}\bigg| + C \|\xi\|_{H^1}\|u_\xi - u\|_{H^1_0}\|\eta\|_{H^1_0}.
$$
Due to the differentiability of $h_\gamma$, we obtain
\begin{equation} \label{eq: bound on eta}
\|\eta\|_{H^1_0} \leq C \left( \|u_\xi - u\|^2_{W^{1,p }} + \|\xi\|_{H^1}\|u_\xi - u\|_{H^1_0} \right),
\end{equation}
for all $p > 2$ and some constant $C > 0$. Thanks to \cite[Thm.~1]{Groeger1989}, there is some $\hat p>2$ such that
\begin{equation} \label{eq:groeger bound}
\|u_\xi - u\|_{W^{1,\hat p}} =O(\|\xi\|_{H^1}).
\end{equation}
From the latter and estimates \eqref{eq: a priori estimates of state and lin. eq}, it then follows that
$\|\eta\|_{H^1_0} = O(\|\xi\|^2_{H^1})$. The last relation ensures the Fr\'echet differentiability of $G$ and $z^\xi_\lambda = G'(\lambda)\xi$.
\end{proof}

A second-order differentiability result for the solution mapping can also be obtained under certain regularity assumptions on the data. The second derivative is used in the proof of second order sufficient optimality conditions and, in its discretized version, for the convergence analysis of the proposed Newton type algorithms.
\begin{theorem}\label{second_frechet_diff_G}
If $f \in L^\infty(\Omega)$ and $u(\lambda) \in C^{1,\beta}(\Omega)$, for some $\beta \in (0,1)$, and there exists $\hat p >4$ such that
\begin{equation} \label{eq: enhanced Groeger}
\|u_\zeta - u\|_{W^{1,\hat p}} \leq C \|\zeta\|_{H^1}, \text{ for any }\zeta \in H^1(\Omega),
\end{equation}
then $G$ is twice Fr\'echet differentiable and its second derivative, in directions $(\xi,\zeta)$, is given by $w_\lambda^{(\xi,\zeta)}\in H_0^1(\Omega)$, solution of
\begin{multline}
\mu\big(Dw_\lambda^{(\xi,\zeta)} , Dv\big) + \big(h'_\gamma(Du(\lambda))D w_\lambda^{(\xi,\zeta)}, Dv\big) + 2\int_\Omega \lambda w_\lambda^{(\xi,\zeta)} v \,dx \\
+ \big(h''_\gamma(Du(\lambda)) [Dz^\xi_\lambda, D z_\lambda^\zeta] , Dv \big)
+ 2\int_\Omega \zeta z^\xi_\lambda v \,dx + 2\int_\Omega \xi z^\zeta_\lambda v \,dx = 0, \quad \forall v\in H_0^1(\Omega).
\end{multline}
\end{theorem}
\begin{remark}
The H\"older continuity assumption on the gradient of $u(\lambda)$ and estimate \eqref{eq: enhanced Groeger} may be proved under some hypothesis on the domain and the data (see \cite[Thm.~2.2]{DelosReyesDhamo} and \cite[Thm.~1]{Groeger1989}, repectively).
\end{remark}

\begin{proof}[Proof of Theorem \ref{second_frechet_diff_G}]
If $f \in L^\infty(\Omega)$ and $u(\lambda) \in C^{1,\beta}(\Omega)$, we obtain from elliptic regularity theory (see, e.g., \cite{Tro87}) that
\begin{equation} \label{eq: W1r a priori estimate lin. state}
\|z^\xi_\lambda\|_{W^{1,s}} \leq C \|\xi\|_{H^1}, \text{ for any }s>2,
\end{equation}
and
\begin{align*}
\| \eta \|_{W^{1,r}} & \leq C \left( \| \eta \|_{H_0^1} + \|\xi (u_\xi-u)\|_{L^r} + \| h_\gamma(Du_\xi) - h_\gamma(Du) - h'_\gamma(Du)D(u_\xi - u) \|_{L^r} \right)\\
& \leq C \left( \| \xi \|_{H^1}^2 + \|\xi\|_{H^1} \|u_\xi-u\|_{H_0^1} + \|u_\xi-u\|_{W^{1,r}}^2 \right),
\end{align*}
where $C>0$ stands for a generic constant and $r \in (2,\hat p)$. Thanks to estimates \eqref{eq: a priori estimates of state and lin. eq} and \eqref{eq: enhanced Groeger}, we then obtain that
\begin{equation} \label{eq: W1r a priori estimate derivative}
\| \eta \|_{W^{1,r}} \leq C_r \|\xi\|_{H^1}, \text{ for }r \in (2,\hat p).
\end{equation}

For $\xi, \zeta \in H^1(\Omega)$, we denote by $w^{(\xi, \zeta)}_\lambda$ the solution of the following equation:
\begin{multline}\label{eq_2nd_G}
\mu\big(Dw, Dv\big) + \big(h'_\gamma(Du)Dw, Dv\big) + 2\int_\Omega \lambda w v \, dx\\
+ \big(h''_\gamma(Du) [Dz^\xi_\lambda, D z_\lambda^\zeta] , Dv \big)
+ 2\int_\Omega \zeta z^\xi_\lambda v \, dx + 2\int_\Omega \xi z^\zeta_\lambda v \, dx= 0, \forall v\in H_0^1(\Omega).
\end{multline}
Existence and uniqueness of $w^{(\xi, \zeta)}_\lambda$ follows in a standard manner from the Lax-Milgram theorem.

Let now $\lambda_\zeta:=\lambda + \zeta$ and let $z^\xi_{\lambda_\zeta} := G'(\lambda_\zeta)\xi$, with $u_\zeta$ the solution to (\ref{eq00_1b}) corresponding to $\lambda_\zeta$. Taking the difference between (\ref{equ_g_gamma_d}) for $z^\xi_\lambda$ and $z^\xi_{\lambda_\zeta}$, we get
\begin{multline}\label{eq_sub_z}
\mu\big(D(z^\xi_{\lambda_\zeta} - z^\xi_\lambda), Dv\big)  + \big(h'_\gamma(Du)D(z^\xi_{\lambda_\zeta} - z^\xi_\lambda), Dv\big) + 2\int_\Omega \lambda (z^\xi_{\lambda_\zeta} - z^\xi_\lambda) v \, dx \\
+ \big(\big[h'_\gamma(Du_\zeta) - h'_\gamma(Du)\big]Dz^\xi_{\lambda_\zeta} , Dv \big)+ 2\int_\Omega \zeta z^\xi_{\lambda_\zeta} v \, dx+ 2\int_\Omega \xi(u_\zeta - u)v \, dx= 0, \forall v \in H_0^1(\Omega).
\end{multline}
Testing (\ref{eq_sub_z}) with $v =z^\xi_{\lambda_\zeta} - z^\xi_\lambda $, we get
\begin{multline}\|z^\xi_{\lambda_\zeta} - z^\xi_\lambda\|^2_{H^1_0} \leq C \bigg\{\bigg|\big(\big[h'_\gamma(Du_\zeta) - h'_\gamma(Du)\big]Dz^\xi_\lambda , D(z^\xi_{\lambda_\zeta} - z^\xi_\lambda) \big)\bigg|\\
+ \bigg|\int_\Omega \zeta z^\xi_\lambda (z^\xi_{\lambda_\zeta} - z^\xi_\lambda) \, dx \bigg| + \bigg|\int_\Omega \xi(u_\zeta - u)(z^\xi_{\lambda_\zeta} - z^\xi_\lambda) \, dx \bigg|\bigg\}.
\end{multline}
From the the Lipschitz properties of $h_\gamma'(\cdot)$ the last relation yields
$$
\|z^\xi_{\lambda_\zeta} - z^\xi_\lambda\|_{H^1_0}   \leq C \left( \big\|u_\zeta - u\big\|_{W^{1,\hat p}} \|z^\xi_\lambda\|_{W^{1,\hat r}} + \|\zeta\|_{H^1} \|z^\xi_\lambda\|_{H^1_0} + \|\ \xi\|_{H^1}\|u_\zeta - u\|_{H^1_0} \right),
$$
with $\hat r$ such that $1/\hat p+1/\hat r \leq 1/2$. Considering \eqref{eq: W1r a priori estimate lin. state} and \eqref{eq:groeger bound}, then the following estimate holds
\begin{equation} \label{eq: bound difference lin. states}
\|z^\xi_{\lambda_\zeta} - z^\xi_\lambda\|_{H^1_0}   \leq  C \|\zeta\|_{H^1} \|\xi\|_{H^1}.
\end{equation}
Again, thanks to elliptic regularity theory,
\begin{equation} \label{eq: W1p bound difference lin. states}
\|z^\xi_{\lambda_\zeta} - z^\xi_\lambda\|_{W^{1,\hat q}}   \leq  C_p \|\zeta\|_{H^1} \|\xi\|_{H^1}, \quad \text{ for }\hat q=\frac{\hat r \hat p}{\hat r+ \hat p}>2.
\end{equation}
In particular, we may choose $\hat r \geq \frac{4 \hat p}{\hat p-4}$, which yields $\hat q \geq 4$.

By setting $\tau := z^\xi_{\lambda_\zeta} - z^\xi_\lambda - w^{(\xi, \zeta)}_\lambda$ and subtracting  (\ref{eq_2nd_G}) from (\ref{eq_sub_z}), we get that
\begin{multline*}
\mu\big(D\tau, Dv\big)  + \big(h'_\gamma(Du)D\tau, Dv\big) + 2\int_\Omega \lambda \tau v \, dx= \\
- \big(\big[h'_\gamma(Du_\zeta) - h'_\gamma(Du)\big] D(z^\xi_{\lambda_\zeta} - z^\xi_\lambda) , Dv \big) - 2\int_\Omega \zeta (z^\xi_{\lambda_\zeta} - z^\xi_\lambda) v \, dx- 2\int_\Omega \xi(u_\zeta - u- z^\zeta_\lambda)v \, dx\\
- \bigg(h'_\gamma(Du_\zeta)Dz^\xi_\lambda - h'_\gamma(Du)Dz^\xi_\lambda - h''_\gamma(Du) [Dz^\xi_\lambda, D z_\lambda^\zeta], Dv \bigg), \quad \forall v \in H_0^1(\Omega).
\end{multline*}
Testing the last equation with $v = \tau$ and using the ellipticity of the terms on the left hand side, we obtain that
\begin{multline}\label{eq_tau_xi_zeta}
\|\tau\|_{H^1_0}  \leq C \bigg\{ \left\| \big[h'_\gamma(Du_\zeta) - h'_\gamma(Du)\big] D(z^\xi_{\lambda_\zeta} - z^\xi_\lambda) \right\|_{L^2} + \left\| \zeta (z^\xi_{\lambda_\zeta} - z^\xi_\lambda) \right \|_{L^2}+ \\
+ \left\| \xi(u_\zeta - u- z^\zeta_\lambda) \right\|_{L^2}+ \left\| h'_\gamma(Du_\zeta)Dz^\xi_\lambda - h'_\gamma(Du)Dz^\xi_\lambda - h''_\gamma(Du) [Dz^\xi_\lambda, D z_\lambda^\zeta] \right\|_{L^2} \bigg\}.
\end{multline}

For the first term on the right hand side, thanks to the Lipschitz continuity of $h_\gamma'$ and estimate \eqref{eq: W1p bound difference lin. states}, we get that
\begin{align*}
\left\| \big[h'_\gamma(Du_\zeta) - h'_\gamma(Du)\big] D(z^\xi_{\lambda_\zeta} - z^\xi_\lambda) \right\|_{L^2} & \leq \left\| h'_\gamma(Du_\zeta) - h'_\gamma(Du) \right\|_{L^{\hat p}} \left\| z^\xi_{\lambda_\zeta} - z^\xi_\lambda \right\|_{W^{1,\hat q}}\\
& \leq L \|u_\zeta-u\|_{W^{1,\hat p}} \left\| z^\xi_{\lambda_\zeta} - z^\xi_\lambda \right\|_{W^{1,\hat q}}\\
& \leq C \|\zeta\|^2_{H^1} \|\xi\|_{H^1}.
\end{align*}

Since the solution operator has been proved to be Fr\'echet differentiable, it follows that
$\|u_\zeta - u- z^\zeta_\lambda\|_{H^1_0} = o(\|\zeta\|_{H^1})$
and, thus,
$$\left\| \xi(u_\zeta - u- z^\zeta_\lambda)\right\|_{L^2} \leq C \|\xi\|_{H^1} o(\|\zeta\|_{H^1}).$$

From \eqref{eq: bound difference lin. states} it also follows that
$$\left\| \zeta (z^\xi_{\lambda_\zeta} - z^\xi_\lambda) \right\|_{L^2} \leq C \|\zeta\|^2_{H^1} \|\xi\|_{H^1}.$$

For the last term on the right hand side of \eqref{eq_tau_xi_zeta}, we obtain that
\begin{multline*}\left\| \left( h'_\gamma(Du_\zeta) - h'_\gamma(Du) - h''_\gamma(Du) D z^\zeta_\lambda \right) D z_\lambda^\xi \right \|_{L^2} \leq
\left\|h''_\gamma(Du) D (u_\zeta-u-z_\lambda^\zeta) \right \|_{L^r} \|Dz^\xi_\lambda\|_{L^s}\\
+\left\| h'_\gamma(Du_\zeta) - h'_\gamma(Du) - h''_\gamma(Du) D(u_\zeta - u) \right\|_{L^r} \|Dz^\xi_\lambda\|_{L^s},
\end{multline*}
where $1/r+1/s=1/2$ and $r \in (2, \hat p)$. Taking into account estimates \eqref{eq: enhanced Groeger}, \eqref{eq: W1r a priori estimate lin. state} and \eqref{eq: W1r a priori estimate derivative} we get that
\begin{multline*}\left\| \left( h'_\gamma(Du_\zeta) - h'_\gamma(Du) - h''_\gamma(Du) D z^\zeta_\lambda \right) D z_\lambda^\xi \right \|_{L^2} \leq C \|\xi\|_{H^1} \left( o(\| \zeta \|_{H^1}) + o(\|u_\zeta-u\|_{W^{1,\hat p}}) \right).
\end{multline*}
Now plugging the last estimates into (\ref{eq_tau_xi_zeta}) and using \eqref{eq: enhanced Groeger}, we get that
$$\|\tau\|_{H^1_0}  \leq C \|\xi\|_{H^1} o(\|\zeta\|_{H^1}) .$$
The last relation ensures the twice differentiability of $G$ and we also have that $w^{(\xi, \zeta)}_\lambda = G''(\lambda)[\xi, \zeta]$.
\end{proof}

\subsection{Optimality conditions}
\label{sub:Optimality conditions}

Based on the differentiability properties of the solution operator, a first order optimality system characterizing the optimal weight function is derived next. The solutions to the optimality system are stationary points, which may or may not correspond to local optimal solutions of \eqref{eq00_1}. To verify that a stationary point is actually a minimizer, second order sufficient optimality conditions are investigated thereafter.
\begin{theorem}\label{thr_1}
Let $(u, \lambda)\in  H^1_0(\Omega)\times  \iV$ be an optimal solution for $(\ref{eq00_1})$. Then there exist $p \in  H^1_0(\Omega)$ and $\vartheta \in L^2(\Omega)$ such that the following optimality system holds (in weak sense):
{\allowdisplaybreaks
\begin{subequations}\label{eq2.30}
\begin{align}
- \mu \Delta u &- \Div ~q + 2 \lambda (u - f) = 0 && \text{ in }\Omega, \label{2.30a} \\
u&=0 &&\text{ on }\Gamma,\\
q &= h_\gamma(Du) && \text{a.e. in }\Omega,\label{2.30a1}\\
- \mu \Delta p &- \Div ~z + 2(\lambda p + u-u^\dag) = 0 && \text{ in }\Omega, \label{2.30b}\\
p&=0 &&\text{ on }\Gamma,\\
z& = h'_\gamma(Du)^*Dp && \text{a.e. in }\Omega, \label{2.30b1}\\
\vartheta &=-\beta \Delta \lambda+ \beta \lambda + (u - f)p && \text{ in }\Omega,\label{2.30c} \\
\frac{\partial \lambda}{\partial \vec{n}} & = 0 &&\text{ on }\Gamma,\\
\lambda &\geq 0, \quad \vartheta \geq 0, \quad \vartheta \, \lambda=0 && \text{a.e. in }\Omega.\label{2.30e}
\end{align}
\end{subequations}}
\end{theorem}
\begin{proof}
Since the solution operator is differentiable, it follows, using the reduced cost functional
\begin{equation} \label{eq: VI existence proof of multipliers}
\mathcal J(\lambda)= \|u(\lambda)- u^\dag\|^2_{L^2} + \beta \|\lambda\|^2_{H^1(\Omega)},
\end{equation}
that
\begin{equation}
\mathcal{J}'(\lambda)(\xi- \lambda)= (u(\lambda)- u^\dag,u'(\lambda)(\xi-\lambda))+ \beta (\lambda, \xi- \lambda)_{H^1} \geq 0, \qquad \forall \xi \in \iV.
\end{equation}
Introducing $p \in H_0^1(\Omega)$ as the unique weak solution of the adjoint equations \eqref{2.30b}-\eqref{2.30b1} and using the linearised equation \eqref{equ_g_gamma_d}, we obtain that
\begin{align*}
2(u- u^\dag, u')&= - \mu (Dp, D u')-\int_\Omega h_\gamma'(Du)^*Dp \cdot Du' \, dx-2 \int_\Omega \lambda u' p \, dx\\
&= 2 \int_\Omega p (u-f) (\xi-\lambda) \, dx,
\end{align*}
where we used the notation $u':=u'(\lambda)(\xi-\lambda)$.
Replacing the last term in \eqref{eq: VI existence proof of multipliers}, we get that
\begin{equation} \label{eq: VI2 existence proof of multipliers}
\beta (\lambda, \xi-\lambda)_{H^1}+ \int_\Omega p (u-f) (\xi- \lambda) \, dx \geq 0, \quad \forall \xi \in \iV.
\end{equation}

Inequality \eqref{eq: VI2 existence proof of multipliers} corresponds to an obstacle type problem with unilateral bounds. Thanks to regularity results for this type of problems (see \cite[Thm.5.2, ~p.294]{Tro87}), it follows that $\lambda \in H^2(\Omega)$ (if $f \in L^p(\Omega)$ for some $p>2$) and, therefore, we may define
$$\vartheta := - \beta \Delta  \lambda  + \beta \lambda + (u - f) p \in L^2(\Omega).$$
Integrating by parts in \eqref{eq: VI2 existence proof of multipliers} we then obtain that $\big(\vartheta, \xi - \lambda\big)_{L^2}\geq 0$. From the latter and the sign of $\lambda$, we finally get that
\begin{align}
\lambda \geq 0,\quad \vartheta \geq 0,\quad \vartheta ~ \lambda  &= 0\quad  a.e.\quad \Omega.\label{ieq_2.5b}
\end{align}
\end{proof}

\begin{remark} \label{rem: extra regularity of adjoint}
If $u^\dag \in L^\infty(\Omega)$ and $u(\lambda) \in C^{1,\beta}(\Omega)$, it follows from elliptic regularity theory (see, e.g., \cite{Tro87}) that the adjoint state has the extra regularity $p \in W^{1,q}(\Omega)$, for all $q \in (2,+\infty)$, and
\begin{equation} \label{eq: W1r a priori estimate adjoint state}
\| p \|_{W^{1,q}} \leq C_q \|u-u^\dag\|_{L^\infty}.
\end{equation}
\end{remark}

The complementarity condition (\ref{ieq_2.5b}) can also be reformulated as the following nonsmooth equation:
$$\vartheta = \max (0, \vartheta - \alpha \lambda ), \text{ for any }\alpha > 0,$$
where the $\max$ operation has to be understood in an almost everywhere sense. By choosing $\alpha = \beta$ and using \eqref{2.30c} one gets
\begin{equation}\label{ineq_lambda}
- \beta \Delta  \lambda  + \beta \lambda + (u - f)p - \max ( 0, - \beta \Delta  \lambda  + (u- f) p ) = 0.
\end{equation}
Altogether, we obtain the following system for $y=(u, q, p, z, \lambda)$
\begin{equation}\label{s3.1_strong_form}
F(y) = \left(\begin{array}{c}
- \mu \Delta u - \Div\hspace{0.25em} q + 2\lambda (u - f) \\
h_\gamma(Du) - q \\
- \mu \Delta  p - \Div\hspace{0.25em}  z + 2\lambda p + 2(u-u^\dag) \\
h'_\gamma(Du)^*Dp - z\\
- \beta \Delta  \lambda  + \beta \lambda + (u - f)p-\max \big(0, - \beta \Delta  \lambda  + (u - f)p\big)
\end{array} \right)=  0,
\end{equation}
where $F:  V \to  W$ with $V :=  H^1_0(\Omega) \times L^2(\Omega) \times  H^1_0(\Omega) \times L^2(\Omega) \times  H^1(\Omega)$ and $W :=  H^{-1}(\Omega) \times L^2(\Omega) \times H^{-1}(\Omega) \times L^2(\Omega) \times  H^{1}(\Omega)'$. The last equation in \eqref{s3.1_strong_form} is complemented with homogeneous Neumann boundary condition for $\lambda$.\\

As mentioned previously, sufficient optimality conditions are important in order to verify that a given stationary point is indeed a minimizer of the original optimization problem. Thanks to the differentiability properties of the solution mapping (see Theorem \ref{second_frechet_diff_G}), we can derive a second-order sufficient optimality condition. To state it, let us start by computing the second derivatives of $J(u, \lambda)$ and the state equation operator $e(u, \lambda)$ defined in \eqref{eq00_1b}. For $(u, \lambda)\in H_0^1(\Omega)\times H^1(\Omega)$ and for all $w, \eta \in H^1_0(\Omega), l\in H^1(\Omega)$, we have:
\begin{subequations}\label{2nd_de_JE}
\begin{align}
& e_{\lambda\lambda}(u, \lambda)  = 0\label{2nd_de_c}\\
& \langle e_{u\lambda}(u, \lambda)[w,l], v \rangle_{H^{-1}, H^1_0} =2\int_\Omega wlv \, dx \quad \forall v\in H^1_0(\Omega)\\
& \langle e_{uu}(u, \lambda)[w,\eta], v \rangle_{H^{-1}, H^1_0} = \int_\Omega h''(Du)[Dw,D \eta] \cdot Dv \, dx, \quad \forall v\in H^1_0(\Omega).
\end{align}
\end{subequations}
Note that for any fixed $\lambda \in H^1(\Omega)$ and $u \in H^1_0(\Omega)$, we also get
\begin{equation}\label{eq2.5}
\langle e_u(u, \lambda)w, v \rangle_{H^{-1}, H^1_0} = \mu \big(Dw, Dv\big)_{L^2} + \big(h'_\gamma(Du) Dw, Dv\big)_{L^2} + 2\int_\Omega{\lambda w v} \, dx,
\end{equation}
for all $v\in  H^1_0(\Omega)$. Now let $a := 1 - \frac{1}{2\gamma}$ and $b := 1 + \frac{1}{2\gamma}$,  and let us introduce the sets
\begin{equation}\label{s3.set}
\begin{aligned}
\iA^\gamma(u) &:= \big\{\in \Omega: \gamma |Du(x)| \geq b\big\},\\
\iS^\gamma(u) &:= \big\{x \in \Omega: a < \gamma |Du(x)| < b\big\},\\
\iI^\gamma(u) &:= \big\{x \in \Omega: \gamma |Du(x)|\leq a\big\},
\end{aligned}
\end{equation}
and $t_1(u):=\frac{\gamma}{2}\big(\gamma |Du| - a\big) = \frac{\gamma}{2}\big(\gamma |Du| - 1 + \frac{1}{2\gamma}\big)$; $t_2(u) = \gamma |Du| -1 - \frac{3}{2\gamma}$. For all $z \in H^1_0(\Omega)$, we get the following expressions for the derivatives of $h_\gamma$:
\begin{equation}\label{s3.3_deriva_h}
\begin{aligned}
h'_\gamma(Du) Dz &=
\chi_{\iA^\gamma(u)}\bigg\{\dfrac{Dz}{|Du|} - \dfrac{\lag Du, Dz\rag}{|Du|^3}Du\bigg\} +\gamma \chi_{\iI^\gamma(u)}\big(Dz\big)\\
 &+ \chi_{\iS^\gamma(u)}\bigg\{\frac{\gamma}{2} Dz + \gamma^2 \big(\gamma|Du| -1\big)\bigg[ 2\gamma^2 t_1(u)t_2(u) -1\bigg] \frac{\lag Du, Dz\rag}{|Du|^2} Du \\
&+ \bigg[\dfrac{2\gamma-1}{4\gamma} - \dfrac{\gamma t_1(u)t_2(u)}{2} +\dfrac{\gamma^3t^2_1(u)t^2_2(u)}{2}\bigg]
\bigg( \frac{Dz}{|Du|} - \frac{\lag Du, Dz\rag}{|Du|^3} Du\bigg)\bigg\} \\
\end{aligned}
\end{equation}
and
\begin{equation}\label{s3.4_deriva_h_u}
\begin{aligned}
h''_\gamma(Du)[Dp,Dz] & = \chi_{\iA^\gamma(u)}\Phi(Du, Dp) Dz \\
+ \chi_{\iS^\gamma(u)}&\bigg\{\bigg[\frac{\gamma}{2}t_1(u)t_2(u)\bigg( 4\gamma^3 |Du| \big(\gamma|Du| - 1\big)- \gamma^2t_1(u)t_2(u) + 1 \bigg)  \\
&- \bigg(\gamma^3 |Du|^2 - \gamma^2|Du| +\frac{1}{2} - \frac{1}{4\gamma}\bigg)\bigg]\Phi(Du, Dp)Dz  \\
&+6\gamma^5  t_1(u)t_2(u)\frac{\lag Du, Dp\rag (Du Du^T)}{|Du|^3} Dz \bigg\},
\end{aligned}
\end{equation}
with the operator $$\Phi(Du, Dp):= \dfrac{3\lag Du, Dp\rag(Du Du^T)}{|Du|^5}-\dfrac{(Dp Du^T)}{|Du|^3} - \dfrac{(Du Dp^T)}{|Du|^3} - \dfrac{\lag Du, Dp\rag}{|Du|^3}.$$

We also define the cone of critical directions by
\begin{equation}\label{def_Kappa}
\iK(\lambda^*) = \left\{l\in H^1(\Omega): l(x)
\begin{cases}
=0\quad\mbox{if}\quad \vartheta(x) \not= 0\\
\geq 0 \quad\mbox{if}\quad \vartheta(x) = 0 \quad\mbox{and}\quad \lambda^*(x) =0
\end{cases}
\right\}.
\end{equation}
Now let us state the second order optimality condition for problem \eqref{eq00_1}. The proof goes along the lines of \cite{Juan_Carlos_Kunisch, Juan_Carlos_Kunisch2}. However, since in our case the control enters in a bilinear way and the PDE has a quasilinear structure, the proof has to be modified accordingly.
\begin{theorem}\label{lem_2nd_order}
Under the same hypotheses of Theorem \ref{second_frechet_diff_G}, let $(u^*, \lambda^*, p^*)$ be a solution of the optimality system \eqref{eq2.30} and suppose that there exists $\rho > 0$ such that
\begin{equation}\label{2nd_order}2\|w\|^2_{L^2} + 2\beta\|l\|^2_{H^1} + \int_\Omega h''(Du^*)[Dw]^2 \cdot Dp^* \, dx+ 4\int_\Omega wlp^* \, dx \geq \rho \|l\|^2_{H^1},
\end{equation}
for every pair $(w, l)\in H^1_0(\Omega)\times \iK(\lambda^*)$, $(w, l)\not= (0, 0)$ which satisfies the linearized equation:
\begin{equation}\label{2nd_cond}\mu\big(Dw, Dv\big)_{L^2} + \big(h'_\gamma(Du^*)Dw, Dv\big)_{L^2} + 2\int_\Omega l(u^*-f)v \, dx + 2\int_\Omega \lambda^* wv \, dx = 0, \forall v \in V.
\end{equation}
Then there exist $\sigma > 0$ and $\tau > 0$ such that
\begin{equation} \label{eq:ssc growth condition}
J(u^*, \lambda^*) + \tau \|\lambda - \lambda^*\|^2_{H^1} \leq J(u, \lambda),
\end{equation}
for every feasible pair $(u, \lambda)$ such that $u = G (\lambda)$ and $\|\lambda - \lambda^*\|_{H^1}\leq \sigma$.
\end{theorem}
\begin{proof}
Suppose that $\lambda^*$ does not satisfy the growth condition \eqref{eq:ssc growth condition}. Then there exists a feasible sequence $\{\lambda_k\}_k\subset H^1(\Omega)$ such that
\begin{align}\label{eq_cons1}
\|\lambda_k - \lambda^*\|_{H^1} &< \frac{1}{k^2}\quad\mbox{and}\\
\label{eq_cons2}
J(u^*, \lambda^*) + \frac{1}{k}\|\lambda - \lambda^*\|^2_{H^1} &> J(u_k, \lambda_k) = \iL(u_k, \lambda_k, p^*)\quad\forall k,
\end{align}
where $u_k = G(\lambda_k)$ and $\iL(u, \lambda, p):= \lag e(u, \lambda), p\rag_{H^{-1},H_0^1} + J(u, \lambda)$. From \eqref{eq: enhanced Groeger} we then get that $u_k \to u^*$ strongly in $W^{1, \hat p},$ with $\hat p >4$. By setting $\rho_k = \|\lambda_k - \lambda^*\|_{H^1}$ and $\zeta_k = \frac{1}{\rho_k}(\lambda_k - \lambda^*)$ it follows that $\|\zeta_k\|_{H^1} = 1$ and therefore we may extract a subsequence, denoted the same, which converges to $\zeta$ weakly in $H^1(\Omega)$.

\noindent \emph{Step 1.} By the mean value theorem we have
\begin{align*}
\iL(u_k, \lambda_k, p^*) + \iL_u(\nu_k, \lambda_k, p^*) (u^* - u_k) &= \iL(u^*, \lambda_k, p^*)\\
&=\iL(u^*, \lambda^*,p^*) + \rho_k \iL_\lambda (u^*, \xi_k, p^*)\zeta_k
\end{align*}
where $\nu_k$, $\xi_k$ are points between $u^*$ and $u_k$, $\lambda^*$ and $\lambda_k$, respectively. From (\ref{eq_cons2}) and $J(u^*, \lambda^*) = \iL(u^*, \lambda^*, p^*)$ it follows that
\begin{equation}\label{eq_04.4}
\iL_\lambda (u^*, \xi_k, p^*)\zeta_k < \frac{1}{k}\|\lambda_k - \lambda^*\|_{H^1} + \frac{1}{\rho_k}\iL_u(\nu_k, \lambda_k, p^*) (u^* - u_k).
\end{equation}
By using again the mean value theorem for the last term on the first variable, we obtain
\begin{align*}\iL_u(\nu_k, \lambda_k, p^*) (u^* - u_k) =& J_u(\nu_k)(u^* - u_k) + \lag p^*, e_u(\nu_k, \lambda_k)(u^* - u_k) \rag_{H^1_0, H^{-1} }\\
=&J_u(\nu_k)(u^* - u_k) + \lag p^*, e_u(u^*, \lambda_k)(u^* - u_k) \rag_{H^1_0, H^{-1} }\\
&+\lag p^*, e_{uu}(u^*,\lambda_k)(\nu_k - u^*)(u^* - u_k) \rag_{H^1_0, H^{-1} }\\
&+\big\lag p^*, \big(e_{uu}(\eta_k, \lambda_k)- e_{uu}(u^*, \lambda_k)\big)(\nu_k - u^*)(u^* - u_k) \big\rag_{H^1_0, H^{-1}},
\end{align*}
where $\eta_k = u^* + t(\nu_k - u^*)$, for some $t\in[0, 1]$. From (\ref{eq2.5}) and the optimality system (\ref{eq2.30}) it follows that
\begin{align*}\lag p^*, e_u(u^*, \lambda_k)(u^* -& u_k) \rag_{H^1_0, H^{-1}}\\
 = &\lag p^*, e_u(u^*, \lambda^*)(u^* - u_k) \rag_{H^1_0, H^{-1}} + 2\int_\Omega(\lambda_k - \lambda^*)(u^* - u_k)p^* \, dx\\
=&-  J_u(u^*) (u^* - u_k)  + 2\int_\Omega(\lambda_k - \lambda^*)(u^* - u_k)p^* \, dx.
\end{align*}
Hence, from the Lipschitz continuity and the boundedness of $h_\gamma''$, and the extra regularity of $p$ (see Remark \ref{rem: extra regularity of adjoint}), we get
\begin{align*}
\big|\iL_u(\nu_k, \lambda_k, p^*) (u^* - u_k)\big| \leq& \|J_u(\nu_k) -  J_u(u^*)\|_{H^{-1}}\|u^* - u_k\|_{H^1_0} \\
&+ 2\|\lambda_k - \lambda^*\|_{L^3}\|u^* - u_k\|_{L^3}\|p^*\|_{L^3} \\
&+ L^\gamma_1 \| p^*\|_{H^1_0} \|\nu_k - u^*\|_{W^{1, \hat p}}\|u^* - u_k\|_{W^{1, \hat p}}\\
 &+ L^\gamma_2 \| p^*\|_{W^{1,4}} \|\nu_k - u^*\|^2_{W^{1, \hat p}} \|u^* - u_k\|_{W^{1, \hat p}}.
\end{align*}
Due to the quadratic cost and the convergence $\zeta_k \rightharpoonup \zeta$, $\xi_k \rightarrow \lambda^*$ in $H^1(\Omega)$ and $u_k \rightarrow u^*$ in $W^{1,4}(\Omega)$, from (\ref{eq_04.4}) it follows that
$$\iL_\lambda (u^*,\lambda^*, p^*)\zeta = \underset{k\rightarrow \infty}{\lim}\iL_\lambda(u^*, \xi_k, p^* )\zeta_k \leq 0.$$
On the other hand, since $\lambda_k(x) \geq 0$ a.e in $\Omega$, it follows that
\begin{equation}\label{eq_4.6}
\iL_\lambda (u^*, \lambda^*, p^*)\zeta_k = \rho_k\iL_\lambda (u^*, \lambda^*, p^*)(\lambda_k - \lambda^*)\geq 0.
\end{equation}
Since $\zeta_k \rightharpoonup \zeta$ one gets $\iL_\lambda (u^*, \lambda^*, p^*)\zeta = \underset{k\rightarrow \infty}{\lim} \iL_\lambda (u^*, \lambda^*, p^*)\zeta_k \geq 0$. Altogether we obtain that $\iL_\lambda (u^*, \lambda^*, p^*)\zeta = 0$.\\

\noindent \emph{Step 2.} Now we will show that $\zeta \in \iK(\lambda^*)$. The set
$$\left\{v\in H^1(\Omega): v(x)\geq 0 \quad\mbox{if}\quad \vartheta(x) = 0 \quad\mbox{and}\quad \lambda^*(x) =0\right\}$$
is convex and closed, hence it is weakly sequentially closed. Since $\lambda_k$ is feasible, then for each $k$, $\zeta_k$ belongs to this set and, consequently, $\zeta$ also does. From (\ref{2.30e}) it follows that $\vartheta(x)\zeta(x) \geq 0$ a.e in $\Omega$, which implies
$$
0= \iL_\lambda (u^*, \lambda^*, p^*)\zeta = \beta \big(\lambda^*, \zeta\big)_{ H^1} + \int\limits_\Omega{(u^* - f)p^* \zeta}= \int_\Omega\vartheta \zeta \, dx= \int_\Omega|\vartheta \zeta| \, dx.
$$
It follows that $\zeta(x) = 0$ if $\vartheta (x) \not=0$ and therefore $\zeta \in \iK(\lambda^*)$.\\

\noindent \emph{Step 3 ($\zeta = 0$).} Using a Taylor expansion of the Lagrangian $\iL$ at $(u^*, \lambda^*, p^*)$ we have
\begin{equation}\begin{aligned}\iL(u_k, \lambda_k, p^*) = &\iL(u^*, \lambda^*, p^*) + \rho_k\iL_\lambda (u^*, \lambda^*, p^*)\zeta_k + \frac{\rho_k^2}{2}\iL_{\lambda\lambda}(u^*, \lambda^*, p^*)\zeta_k^2 \\
&+\rho_k\iL_{u\lambda}(u^*, \lambda^*, p^*)(u_k - u^*)\zeta_k + \frac{1}{2}\iL_{uu}(\nu_k, \lambda^*, p^*)(u_k - u^*)^2,
\end{aligned}\end{equation}
where $\nu_k$ is an intermediate point between $u_k$ and $u^*$. Therefore, thanks to the bilinear control structure,
\begin{equation}\label{eq_2.21}\begin{aligned}\rho_k\iL_\lambda (u^*, \lambda^*, p^*)\zeta_k &+ \frac{\rho_k^2}{2}\iL_{\lambda\lambda}(u^*, \lambda^*, p^*)\zeta_k^2 +\rho_k\iL_{u\lambda}(u^*, \lambda^*, p^*)(u_k - u^*)\zeta_k \\
&+ \frac{\rho_k^2}{2}\iL_{uu}(u^*, \lambda^*, p^*)\bigg(\frac{u_k - u^*}{\rho_k}\bigg)^2\\
= \iL(u_k, \lambda_k, p^*) &- \iL(u^*, \lambda^*, p^*)\\
 &+ \frac{\rho_k^2}{2}\bigg[\iL_{uu}(u^*, \lambda^*, p^*) - \iL_{uu}(\nu_k, \lambda^*, p^*)\bigg]\bigg(\frac{u_k-u^*}{\rho_k}\bigg)^2.
\end{aligned}\end{equation}
Moreover, from (\ref{eq_cons2}) it follows that
\begin{equation}\label{eq_2_21}
\iL(u_k, \lambda_k, p^*) - \iL(u^*, \lambda^*, p^*) < \dfrac{\rho_k^2}{k}.\end{equation}
From the properties of $G$, we have that $\|\frac{u_k - u^*}{\rho_k}\|_{W^{1,4}} = \frac{\|G(\lambda_k) - G(\lambda^*)\|_{W^{1,4}}}{\|\lambda_k - \lambda^*\|_{H^1}}$ is bounded. Hence, from $\lambda_k \rightarrow \lambda^*$, $\|\zeta_k\|_{H^1} = 1$ and by \eqref{eq: enhanced Groeger} we obtain
\begin{equation}
\begin{aligned}\bigg| \bigg[&\iL_{uu}(u^*, \lambda^*, p^*) - \iL_{uu}(\nu_k, \lambda^*, p^*)\bigg]\bigg(\frac{u_k-u^*}{\rho_k}\bigg)^2 \bigg|\\
 &\leq L^\gamma_2 \|p^*\|_{W^{1,4}} \|u^*-u_k\|_{W^{1,4}} \bigg\|\frac{u_k - u^*}{\rho_k}\bigg\|^2_{W^{1,4}} \overset{k \rightarrow \infty}{\longrightarrow} 0.
\end{aligned}
\end{equation}
From (\ref{eq_2.21}) it follows that
\begin{multline*}\underset{k\rightarrow \infty}{\lim}\inf \iL_{\lambda\lambda}(u^*, \lambda^*, p^*)\zeta_k^2 + \underset{k\rightarrow \infty}{\lim}\inf \iL_{uu}(u^*, \lambda^*, p^*)\bigg(\frac{u_k-u^*}{\rho_k}\bigg)^2\\ +2 \underset{k\rightarrow \infty}{\lim}\inf \frac{1}{\rho_k}\iL_{u\lambda}(u^*, \lambda^*, p^*)(u_k - u^*)\zeta_k\\
 \leq 2 \underset{k\rightarrow \infty}{\lim}\sup\frac{1}{\rho_k^2}\big[\iL(u_k, \lambda_k, p^*) - \iL(u^*, \lambda^*, p^*)\big] - 2 \underset{k\rightarrow \infty}{\lim}\inf \frac{1}{\rho_k} \iL_\lambda(u^*, \lambda^*, p^*)\zeta_k.
\end{multline*}
Since $\iL_{\lambda\lambda}(u^*, \lambda^*, p^*)\zeta_k^2 = 2\beta\|\zeta_k\|^2_{H^1}$ is weakly lower semi-continuous and from (\ref{eq_4.6}), (\ref{eq_2_21}), the last relation implies
\begin{equation}\label{eq_2.25}
\begin{aligned}
\iL_{\lambda\lambda}(u^*, \lambda^*, p^*)\zeta^2& + \underset{k\rightarrow \infty}{\lim}\inf \iL_{uu}(u^*, \lambda^*, p^*)\bigg(\frac{u_k-u^*}{\rho_k}\bigg)^2\\& + 2 \underset{k\rightarrow \infty}{\lim}\inf \iL_{u\lambda}(u^*, \lambda^*, p^*)\bigg(\frac{u_k - u^*}{\rho_k}\bigg)\zeta_k \leq 2\underset{k\rightarrow \infty}{\lim} \frac{1}{k} = 0.
\end{aligned}
\end{equation}
Let us denote by $\vartheta_{\zeta_k}$ the solution of (\ref{2nd_cond}) associated with $\zeta_k$. Since $\zeta_k \rightharpoonup \zeta$ in $H^1(\Omega)$ and $\|\zeta_k\|_{H^1} = 1$ one gets that $\zeta_k \rightarrow \zeta$ in $L^p(\Omega)$, for all $p \in [1,\infty)$. Hence, from the linearized equation and the continuous invertibility of $e_u(u^*, \lambda^*)$, we have $\vartheta_{\zeta_k} \rightarrow \vartheta_\zeta$ in $W^{1,4}(\Omega)$.

Besides,
\begin{align*}
 &\iL_{uu}(u^*, \lambda^*, p^*)\bigg(\frac{u_k-u^*}{\rho_k}\bigg)^2 =  \iL_{uu}(u^*, \lambda^*, p^*)\bigg(\frac{G(\lambda_k)-G(\lambda^*)}{\|\lambda_k-\lambda^*\|_{H^1}} - \vartheta_{\zeta_k}\bigg)^2 \\
 &+ 2\iL_{uu}(u^*, \lambda^*, p^*)\bigg(\frac{G(\lambda_k)-G(\lambda^*)}{\|\lambda_k-\lambda^*\|_{H^1}} - \vartheta_{\zeta_k}, \vartheta_{\zeta_k}\bigg) + \iL_{uu}(u^*, \lambda^*, p^*)(\vartheta_{\zeta_k})^2
\end{align*}
and
\begin{align*}
 \iL_{u\lambda}(u^*, \lambda^*, p^*)\bigg(\frac{u_k - u^*}{\rho_k}\bigg)\zeta_k = &\iL_{u\lambda}(u^*, \lambda^*, p^*)\bigg(\frac{G(\lambda_k)-G(\lambda^*)}{\|\lambda_k-\lambda^*\|_{H^1}} - \vartheta_{\zeta_k}\bigg)\zeta_k\\
 & + \iL_{u\lambda}(u^*, \lambda^*, p^*)(\vartheta_{\zeta_k},\zeta_k).
\end{align*}
Note that $\vartheta_{\zeta_k}$ also corresponds to the derivative of the control-to-state mapping $G$ at $\lambda^*$ in direction $\zeta_k$. From the differentiability of $G$, it follows that $\frac{G(\lambda_k)-G(\lambda^*)}{\|\lambda_k-\lambda^*\|_{H^1}} - \vartheta_{\zeta_k} \underset{k\rightarrow \infty}{\longrightarrow} 0$ in $W^{1,4}(\Omega)$. Due to the continuity of the bilinear form $\iL_{uu}(u^*, \lambda^*, p^*)$, since $\vartheta_{\zeta_k} \rightarrow \vartheta_\zeta$ and from (\ref{2.30c}-\ref{2.30e}),
we get
\begin{align*}\iL_{\lambda\lambda}(u^*, \lambda^*, p^*)\zeta^2  + 2  \iL_{u\lambda}(u^*, \lambda^*, p^*)(\vartheta_{\zeta},\zeta) + \iL_{uu}(u^*, \lambda^*, p^*) \vartheta_\zeta^2 \leq 2\underset{k\rightarrow \infty}{\lim} \frac{1}{k} = 0.
\end{align*}
Since $\zeta \in \mathcal K(\lambda^*)$, from (\ref{2nd_order}) it then follows that $(\zeta, \vartheta_\zeta)=0$.\\

\noindent \emph{Step 4:} Finally, from $\vartheta_{\zeta_k} \rightarrow \vartheta_\zeta =0$, (\ref{2nd_order}), (\ref{eq_4.6}), (\ref{eq_2_21}) we have
$$\underset{k\rightarrow \infty}{\lim}\sup \rho \|\zeta_k\|_{H^1}^2 \leq \underset{k\rightarrow \infty}{\lim}\sup  \iL_{\lambda\lambda}(u^*, \lambda^*, p^*)\zeta_k^2 \leq 2\underset{k\rightarrow \infty}{\lim}  \frac{1}{k} = 0.$$
Hence, $\zeta_k \rightarrow 0$ in $H^1(\Omega)$, which is in contradiction to $\|\zeta_k\|_{H^1} = 1$.
\end{proof}

\section{Discretization and numerical treatment} \label{secreferences}
\setcounter{lemma}{0}
\setcounter{theorem}{0}
\setcounter{corollary}{0}

In this section we present a numerical strategy for the solution of the optimality system \eqref{s3.1_strong_form}. We start by explaining how the domain is discretized using finite differences and introduce the resulting discrete operators. Due to the size of the problem, an overlapping Schwarz domain decomposition strategy is considered, where the transmission conditions between subdomains are determined in an optimized way. The resulting subdomain finite-dimensional nonlinear systems are then solved by using a semismooth Newton method, for which local superlinear convergence is verified. A further modification of the semismooth Newton algorithm is introduced in order to get a global convergent behaviour.

\subsection{Discretization schemes}
For the image domain, we use a finite differences scheme on a uniform mesh and consider the problem (\ref{s3.1_strong_form}) on the domain $\Omega:= [0, (m-1)h]\times [0, (l-1)h]$, where $h$ denotes the mesh step size, and $w, l\in\IN^*$ depend on the resolution of the input data. In practice, $m$ and $l$ are width and length of the input images $f, u^\dag$ in pixels. In what follows, the notation $\mathbf{u}, \mathbf{q}, \mathbf{p}, \mathbf{z}, \boldsymbol{\lambda}$ is used for the discretized variables that approximate $u, q, p, z, \lambda$ and $F_h$, $\Div_h$, $\Delta_h$ are used for the discrete approximations of $F, \Div, \Delta$, respectively.

In order to approximate the state and adjoint variables, as well as their derivatives, we consider a modified finite differences scheme (see \cite{Muravleva2008}). We define the following grid domains:
\begin{align*}
\Omega_h &= \{x_{ij}:= ((i-1)h, (j-1)h)| i=1,\ldots,m; j=1,\ldots,l\},\\
\Omega^1_h &= \{x_{ij}:= ((i-0.5)h, (j-1)h)| i=1,\ldots,m; j=1,\ldots,l\},\\
\Omega^2_h &= \{x_{ij}:= ((i-1)h, (j-0.5)h)| i=1,\ldots,m; j=1,\ldots,l\}.
\end{align*}
and the corresponding spaces of grid functions:
\begin{align*}
U_h &= \{\mathbf u_{ij}:=u(x_{ij})| x_{ij}\in \Omega_h;\quad u_{i0}=u_{0j}=0; \quad 1\leq i \leq m,\quad 1\leq j \leq l\},\\
\Lambda_h &= \{\boldsymbol \lambda_{ij}:=\lambda(x_{ij})| x_{ij}\in \Omega_h; \quad 1\leq i \leq m,\quad 1\leq j \leq l\},\\
D^1_u &= \{\mathbf v_{ij}:=v(x_{ij})| x_{ij}\in \Omega^1_h; \quad 1\leq i \leq m,\quad 1\leq j < l\},\\
D^2_u &= \{\mathbf w_{ij}:=w(x_{ij})| x_{ij}\in \Omega^2_h; \quad 1\leq i < m,\quad 1\leq j \leq l\}.
\end{align*}
Therefore, $\mathbf{u}, \mathbf{p}\in U_h$, $\boldsymbol{\lambda}\in \Lambda_h$ and $\mathbf{q}, \mathbf{z}\in D^1_u \times D^2_u$. We define the operator $D_h$ as follows:
$$D_h:\Lambda_h \longrightarrow D^1_u \times D^2_u, \quad (D_h \mathbf v)_{i, j} = \big((D_{h_{x_1}} \mathbf v)_{i, j},  (D_{h_{x_2}} \mathbf v)_{i, j}\big)$$
where $D_{h_{x_1}}$ and $D_{h_{x_2}}$ are computed by forward differences of the ``inner points''
 $$(D_{h_{x_1}} \mathbf v)_{i, j} := \frac{\mathbf v_{i+1, j} - \mathbf v_{i, j}}{h},\quad (D_{h_{x_2}} \mathbf v)_{i, j} := \frac{\mathbf v_{i, j+1} - \mathbf v_{i, j}}{h}, \quad 1 \leq i < m-1, 1 \leq j < l-1.$$
The discrete Laplacian $\Delta_h: \Lambda_h \rightarrow \Lambda_h$ is computed by using a classical five point stencil. For the homogeneous Neumann boundary conditions for $u, p$ and $\lambda$ we get
$$
\begin{aligned}
\mathbf{u}_{0, j} &= \mathbf{u}_{2, j}, \quad \mathbf{u}_{m+1, j} = \mathbf{u}_{m-1, j}\quad (1\leq j \leq l);\quad \mathbf{u}_{i, 2} = \mathbf{u}_{i, 0}, \quad \mathbf{u}_{i, l+1} = \mathbf{u}_{i, l-1}\quad (1\leq i \leq m)\\
\mathbf{p}_{0, j} &= \mathbf{p}_{2, j}, \quad \mathbf{p}_{m+1, j} = \mathbf{p}_{m-1, j}\quad (1\leq j \leq l);\quad \mathbf{p}_{i, 2} = \mathbf{p}_{i, 0}, \quad \mathbf{p}_{i, l+1} = \mathbf{p}_{i, l-1}\quad (1\leq i \leq m)\\
\boldsymbol{\lambda}_{0, j} &= \boldsymbol{\lambda}_{2, j}, \quad \boldsymbol{\lambda}_{m+1, j} = \boldsymbol{\lambda}_{m-1, j}\quad (1\leq j \leq l);\quad \boldsymbol{\lambda}_{i, 2} = \boldsymbol{\lambda}_{i, 0}, \quad \boldsymbol{\lambda}_{i, l+1} = \boldsymbol{\lambda}_{i, l-1}\quad (1\leq i \leq m).
\end{aligned}
$$
The discrete divergence operator $\Div_h : D^1_u \times D^2_u \rightarrow U_h$ is computed by using backward differences on $\mathbf{q} = (\mathbf{q}^1, \mathbf{q}^2) \in D^1_u \times D^2_u$
 $$(\Div_h \mathbf{q})_{i, j} = \frac{\mathbf{q}^1_{i, j} - \mathbf{q}^1_{i-1, j}}{h} + \frac{\mathbf{q}^2_{i, j} - \mathbf{q}^2_{i, j-1}}{h}.$$

Accordingly, we define the approximation operator $F_h: H_h\rightarrow H'_h$, where $H_h = U_h \times (D^1_u \times D^2_u) \times U_h \times  (D^1_u \times D^2_u)  \times \Lambda_h$ and $H'_h = U_h \times (D^1_u \times D^2_u) \times U_h \times  (D^1_u \times D^2_u)  \times U_h$, and for $\yy = (\mathbf{u}, \mathbf{q}, \mathbf{p}, \mathbf{z}, \boldsymbol{\lambda}) \in H_h $, we obtain the nonlinear system
\begin{equation}\label{s3.1_strong_form_h}
F_h({\yy}) = \left(\begin{array}{c}
- \mu \Delta_h \mathbf{u} - \Div_h\hspace{0.25em} \mathbf{q} + 2\boldsymbol{\lambda} (\mathbf{u} - \mathbf f) \\
h_\gamma(D_h\mathbf{u}) - \mathbf{q} \\
- \mu \Delta_h  \mathbf{p} - \Div_h\hspace{0.25em} \mathbf{z} + 2\boldsymbol{\lambda} \mathbf{p} + 2(\mathbf{u}-\mathbf u^\dag) \\
h'_\gamma(D_h\mathbf{u})^*D_h\mathbf{p} - \mathbf{z}\\
- \beta \Delta_h  \boldsymbol{\lambda}  + \beta \boldsymbol{\lambda} + (\mathbf{u} -\mathbf f)\mathbf{p}-\max \big(0, - \beta \Delta_h  \boldsymbol{\lambda}  + (\mathbf{u} - \mathbf f) \mathbf{p}\big)
\end{array} \right)=  0.
\end{equation}
Above, we used the notation $\mathbf u \mathbf v$ to represent the grid function $(\mathbf{uv})_{ij} = \mathbf u_{ij} \mathbf v_{ij}$ for all $\mathbf u, \mathbf v \in \Lambda_h$ or $\mathbf u, \mathbf v \in D^k_u$ ($k=1,2$). Hereafter, the notations $\lag\cdot, \cdot\rag$ and $\|\cdot\|$ stand for the Euclidian product and norm in $\RR^n$, respectively. Besides, for $\mathbf q=(\mathbf q^1, \mathbf q^2), \mathbf z= (\mathbf z^1, \mathbf z^2) \in D^1_u \times D^2_u$, we denote $(\mathbf q, \mathbf z)_{D^1_u \times D^2_u}:= \lag \mathbf q^1, \mathbf z^1 \rag + \lag \mathbf q^2, \mathbf z^2 \rag$.

\subsection{Schwarz domain decomposition methods}\label{sec:schwarz}
The nonlinear system \eqref{s3.1_strong_form_h}, arising from the discretization of (\ref{s3.1_strong_form}), is of large scale nature, involving the solution of three coupled PDEs per each training pair of images. Even for the case of a single training pair, this task cannot be performed on a standard desktop computer. In the case of larger training sets, the problem becomes much harder, not to mention the increasingly high resolution of the images at hand.

To tackle this problem, we consider the application of Schwarz domain decomposition methods for solving the resulting optimality system. Since our aim is to set up a parallel method based on domain decomposition, we focus on additive Schwarz methods. Once the domain is decomposed, the nonlinear optimality system is solved in each subdomain.

It is well-known that the convergence rate of the Schwarz method is dependent on the size of the overlapping area. In order to improve the convergence rate, a modified version of the method was proposed in \cite{Grande2006, E_Okyere}. To illustrate the main idea, consider the following coupled linear system with an optimality system type structure:
\begin{align*}
-\Delta u + \eta u &= f + \theta p \quad\mbox{in}\quad \Omega, && u = 0\quad\mbox{on}\quad \partial \Omega,\\
-\Delta p + \eta p &= -(u-u_d)\quad\mbox{in}\quad \Omega, && p = 0\quad\mbox{on}\quad \partial \Omega,
\end{align*}
where $\theta, \eta > 0$. The so-called optimized Schwarz method (with two subdomains) works as follows: For $k \geq 0$ and $i,j\in\{1, 2\}$, $i\not=j$, solve
$$\begin{cases}
-\Delta u_i^{k+1} + \eta u_i^{k+1} = f + \mu p_i^{k}\quad\mbox{in}\quad\Omega_i,\\
u_i^{k+1} \big|_{\partial\Omega} = 0, \quad \big(\alpha_i + \partial_{\vec{n}} \big)u_i^{k+1}\big|_{\Gamma_i} = \big(\alpha_i + \partial_{\vec{n}} \big)u_j^{k}\big|_{\Gamma_i}, \vspace{0.3cm}\\
-\Delta p_i^{k+1} + \eta p_i^{k+1} = -(u_i^{k} - u_d) \quad\mbox{in}\quad\Omega_i;\\
p_i^{k+1}\big|_{\partial\Omega} = 0, \quad \big(\alpha_i + \partial_{\vec{n}}\big)p_i^{k+1}\big|_{\Gamma_i} = \big(\alpha_i + \partial_{\vec{n}} \big)p_j^{k}\big|_{\Gamma_i},
\end{cases}$$
where the transmission parameters $\alpha_1, \alpha_2$ are approximated as follows (by zero order approximations)
$$\alpha_1 = \sqrt{\eta},\quad \alpha_2 = -\sqrt{\eta}.$$
For further details on the obtention of $\alpha_1, \alpha_2$ we refer the reader to \cite{Grande2006,E_Okyere}.

In order to obtain the formulas for the transmission parameters of the optimized Schwarz method for our learning problem, we consider the equations for $u$ and $p$ in the optimality system (in strong form) as a coupled system:
$$\begin{aligned}
- &\mu \Delta u - \Div[h_\gamma(Du)]+ 2\lambda (u - f) =0, \\
- &\mu \Delta  p - \Div [h'_\gamma(Du)^*Dp] + 2\lambda p  + 2(u-u^\dag) =0.
\end{aligned}$$
By skipping the terms involving the regularizing function $h_\gamma$ and its derivative, we get again the linear coupled system as in \cite{E_Okyere}. In addition, we consider the gradient equation
$$ -\beta \Delta  \lambda  + \beta \lambda + (u - f)p=0$$
for the functional parameter $\lambda$. We use the common forms of transmission conditions on $\Gamma_1, \Gamma_2$ in the optimized Schwarz method as follows
\begin{equation}\label{s3.11_trans_cond}
\begin{aligned}
\big(\frac{\partial }{\partial \vec{n}} + S_{v_1}^{(u^k, p^k, \lambda^k)}\big)v_1^{k+1} &= \big(\frac{\partial }{\partial \vec{n}} +  S_{v_1}^{(u^k, p^k, \lambda^k)}\big)v_2^k \quad\mbox{on}\quad \Gamma_1;\\
\big(\frac{\partial }{\partial \vec{n}} + S_{v_2}^{(u^k, p^k, \lambda^k)}\big)v_2^{k+1} &= \big(\frac{\partial }{\partial \vec{n}} + S_{v_2}^{(u^k, p^k, \lambda^k)}\big)v_1^k \quad\mbox{on}\quad \Gamma_2,
\end{aligned}
\end{equation}
where the transmission parameters are chosen in a similar way as for the coupled system above (see \cite{E_Okyere}):
$$\begin{aligned}
S_{u_1}^{(u^k, p^k, \lambda^k)} = S_{p_1}^{(u^k, p^k, \lambda^k)}&= \sqrt{\frac{2\lambda_1^n}{\mu}},\quad S_{u_2}^{(u^k, p^k, \lambda^k)} = S_{p_2}^{(u^k, p^k, \lambda^k)} =
 -\sqrt{\frac{2\lambda_2^k}{\mu} } ,\\
S_{\lambda_1}^{(u^k, p^k, \lambda^k)} &= 1,\quad S_{\lambda_2}^{(u^k, p^k, \lambda^k)} = -1.
\end{aligned}$$
Although this choice is merely heuristic, obtained by dismissing the importance of the nonlinear terms, the experimental results are promissing (see Section \ref{sec:Computational experiments} below). A further investigation on the choice of the transmission parameters for optimality systems appears to be of significant interest.

\subsection{Semismooth Newton method}\label{sec:semismooth}
The optimality system \eqref{s3.1_strong_form_h} has a nonlinear nonsmooth structure. Because of this, a Newton method cannot be directly applied. However, the nonsmooth functions involved, in particular the $\max$ operator, have additional properties, which allow to define a generalized Newton step for the solution of the system.

\begin{definition} Let $X, Z$ be Banach spaces and $D\subset X$ be an open set. The mapping $F:D\to Z$ is called Newton differentiable on an open set $U\subset D$ if there exists a mapping $G:U\to \iL(X, Z)$ such that
$$\underset{h\to 0}{\lim}\frac{\|F(x+h) - F(x) - G(x+h)h\|_Z}{\|h\|_X} = 0,\quad h\in X$$
for every $x \in U$. $G$ is called generalized derivative of $F$.
\end{definition}
\noindent We also refer to \cite{Hinter2002, KIto_KKunisch} for a chain rule for Newton differentiable functions.

 \begin{lemma}\label{lm_2.1}Let $F:Y\to Z$ be a Newton differentiable operator with generalized derivative $G$; $y^*$ be a solution of equation $F(y) = 0$ and $U\subset Y$ an open neighborhood containing $y^*$. If for every $y\in U$, $\|[G(y)]^{-1}\|_{\iL(X, Z)}$ is bounded, then the Newton iterations
$$y_{k+1} = y_k - G^{-1}(y_k)F(y_k)$$
converge superlinearly to $y^*$, provided that $\|y_0 - y^*\|_X$ is sufficiently small.
\end{lemma}

In particular, it has been proved (see, e.g., \cite{Hinter2002}) that the mapping $\max(0, \cdot):\RR^n\rightarrow \RR^n$ is Newton differentiable with generalized derivative $G_m: \RR^n\rightarrow \iL\big(\RR^n, \RR^n\big)$ given by
$$(G_m(y))_i = \begin{cases}1\quad\mbox{if}\quad y_i > 0,\\  0\quad\mbox{if}\quad y_i \leq 0\end{cases}.$$

The operator $F_h$ in $(\ref{s3.1_strong_form_h})$ is therefore Newton differentiable and its generalized derivative $\iG_F: H_h\mapsto \iL( H_h, H'_h)$ is given by
\begin{multline}\label{s3.4_jacobian}
\iG_{F_h}(\yy)\delta_{\yy} =\\
\left(\begin{array}{c}
(2 {\boldsymbol{\lambda}} \II -\mu \Delta_h)\delta_{\mathbf u}  -\Div_h \delta_{\mathbf q}   + 2( \mathbf u -  \mathbf f)\delta_{\boldsymbol{\lambda}}\\
h'_\gamma(D_h {\mathbf u})D_h\delta_{\mathbf u} -\delta_{\mathbf q} \\
2\delta_{\mathbf u}  + (2 {\boldsymbol{\lambda}} \II -\mu \Delta_h)\delta_{\mathbf p}   -\Div_h\delta_{\mathbf z}    + 2 \mathbf p\delta_{\boldsymbol{\lambda}} \\
h''_\gamma(D_h {\mathbf u})^*[D_h {\mathbf p}, D_h \delta_{\mathbf u}] + h'_\gamma(D_h {\mathbf u})^* D_h \delta_{\mathbf p} -\delta_{\mathbf z} \\
\mathbf p\delta_{\mathbf u} +  (\mathbf {u} - \mathbf {f})\delta_{\mathbf p}  + \beta(\II - \Delta_h) \delta_{\boldsymbol{\lambda}} - G_m\big((\mathbf u -\mathbf f)\mathbf p- \beta \Delta_h  \boldsymbol{\lambda} \big)\big(\mathbf p\delta_{\mathbf u} +  (\mathbf {u} - \mathbf {f})\delta_{\mathbf p}  - \beta\Delta_h \delta_{\boldsymbol{\lambda}}\big)
\end{array} \right)
\end{multline}
where $\delta_{\yy} = (\delta_{\mathbf u}, \delta_{\mathbf q}, \delta_{\mathbf p}, \delta_{\mathbf z}, \delta_{\boldsymbol{\lambda}})$ and $\II$ stands for the identify. The semi-smooth Newton step is then given by
\begin{equation}\label{ssnewton}
\iG_{F_h}(\yy_k)\delta_\yy = -{F_h}(\yy_k),\quad \yy_{k+1} = \yy_k + \delta_\yy,
\end{equation}
where $F$ and $\iG_{F_h}$ are defined in (\ref{s3.1_strong_form_h}) and (\ref{s3.4_jacobian}), respectively.

For the convergence analysis, we assume that there exists an optimal solution $(\mathbf u^*, \boldsymbol{\lambda}^*)\in U_h\times \Lambda_h$, with $\boldsymbol{\lambda}^* \geq 0$ on $\Omega_h$. The second order condition in Theorem \ref{lem_2nd_order} ensures that a solution of the first order system is also solution of the optimization problem. However, to consider the convergence of the semi-smooth Newton method, we need the following stronger assumption: There exists $ \rho >0$ such that
\begin{multline}\label{2nd_order_b}
2\|w\|^2 + \beta (\|l\|^2 + \|D_hl\|^2)  + \langle h''(D_h {\mathbf u}^*)[D_hw]^2, D_h {\mathbf p}^*\rangle + 4\lag w \, l,\mathbf p^*\rag \geq \rho (\|l\|^2 + \|D_hl\|^2),
\end{multline}
for every pair $(w, l)\in U_h \times \Lambda_h$ that satisfies
$$-\mu \Delta_h w - \Div_h \big(h'_\gamma(D_h {\mathbf u}^*)D_hw \big) + 2 l (\mathbf u^*-\mathbf f) + 2 \boldsymbol{\lambda}^* \ w = 0.$$

Now we consider the mapping $e_{\mathbf u}(\mathbf u, \boldsymbol{\lambda}) \in \iL(U_h, U_h')$ defined by
$$e_{\mathbf u}(\mathbf u, \boldsymbol{\lambda})w = -\mu \Delta_h w - \Div_h \big(h'_\gamma(D_h {\mathbf u})D_hw \big) + 2 \boldsymbol{\lambda} \ w, \quad \forall w \in U_h.$$
From the properties of $h_\gamma'$ it can be verified that $\lag e_{\mathbf u}(\mathbf u, \boldsymbol{\lambda})w, w\rag \geq \lag (2\boldsymbol{\lambda} \II -\mu\Delta_h)w, w\rag, \forall w\in U_h$ and, hence, $e_{\mathbf u}(\mathbf u, \boldsymbol{\lambda})$ is invertible. Moreover, for $\mathbf u\in U_h$ and $\boldsymbol{\lambda} \in \tilde \iK$, there exists $C > 0$ (independent of $\bu$ and $\boldsymbol{\lambda}$) such that for every $\xi \in U_h$, the equation
$$e_{\mathbf u}(\mathbf u, \boldsymbol{\lambda})w = -\mu \Delta_h w - \Div_h \big(h'_\gamma(D_h {\mathbf u})D_hw \big) + 2 \boldsymbol{\lambda} \ w = \xi$$
has a unique solution $w \in U_h$ which satisfies $\|w\| \leq C\|\xi\|$.
If a pair $(w, l)\in U_h \times \Lambda_h$ satisfies the equation
$$e_{\mathbf u}(\mathbf u, \boldsymbol{\lambda})w + e_{\lambda} (\mathbf u, \boldsymbol{\lambda}) l = -\mu \Delta_h w - \Div_h \big(h'_\gamma(D_h {\mathbf u})D_hw \big) +  2 \boldsymbol{\lambda} w + 2 (\mathbf u- \mathbf f)\, l= 0, $$
then $\|w\| \leq C_1(\mathbf u, \boldsymbol{\lambda})\|l\|$, where $C_1(\mathbf u, \boldsymbol{\lambda}) > 0$ is dependent on $(\mathbf u, \boldsymbol{\lambda})$. If we only consider $\mathbf u$ in a bounded neighborhood of $\mathbf u^*$, the last estimate yields
\begin{equation} \label{eq: uniform invertibility bound}
\|w\| \leq C_1\|l\|,
\end{equation}
for some $C_1> 0 $ and for all $w\in U_h, l\in \Lambda_h$ satisfying $e_{\mathbf u}(\mathbf u, \boldsymbol{\lambda})w + e_{\lambda} (\mathbf u, \boldsymbol{\lambda}) l =0$.

\begin{theorem}\label{thrm_2.2} If condition $(\ref{2nd_order_b})$ holds, then the semismooth Newton method applied to $(\ref{s3.1_strong_form_h})$, with generalized derivative $\iG_{F_h}$ defined by $(\ref{s3.4_jacobian})$, converges locally superlinearly to a solution $\yy^* = (\mathbf u^*, \mathbf q^*, \mathbf p^*, \mathbf z^*, \boldsymbol{\lambda}^*)$, provided that $\|\yy_0-\yy^*\|$ is sufficiently small.
\end{theorem}
\begin{proof}
At step $k\geq 0$, we denote $A_k:=\{x_{ij}\in \Omega_h: (\mathbf u - \mathbf f)\mathbf p- \beta \Delta_h  \boldsymbol{\lambda} >0\}$ and $I_k := \Omega_h \setminus A$. $F_h^i$ are the components on the right-hand side, $i=1,..,5$. The $5^{th}$ equation of the system (\ref{ssnewton})
can be expressed as
$$\begin{cases}
\chi_{A_k}\beta\delta_{\boldsymbol{\lambda}} = \chi_{A_k}F_h^5\\
\chi_{I_k}\big\{\mathbf p \delta_{\mathbf u} +  (\mathbf {u} - \mathbf{f}) \delta_{\mathbf p}  + \beta(\II - \Delta_h) \delta_{\boldsymbol{\lambda}}\big\}=\chi_{I_k}F_h^5.
\end{cases}$$
Moreover, since from the $2^{nd}$ and $4^{th}$ equations we obtain an explicit expresssion for $\delta_{\mathbf q}$ and $\delta_{\mathbf z}$, respectively, we may write (\ref{ssnewton}) in equivalent form as
\begin{subequations}\label{reduce_system}
\begin{align}
& (2 \boldsymbol{\lambda}_k-\mu \Delta_h )\delta_{\mathbf u} -\Div_h ~h'_\gamma(D_h {\mathbf u}_k)[D_h \delta_{\mathbf u}]  +  2(\mathbf u_k -  \mathbf f) \delta_{\boldsymbol{\lambda}} = g_1, \label{redu_a}\\
&2 \delta_{\mathbf u} - \Div_h ~h''_\gamma(D_h {\mathbf u}_k)^* [D_h {\mathbf p}_k,D_h \delta_{\mathbf u}] \label{redu_b}\\
& \qquad+ (2 \boldsymbol{\lambda}_k-\mu \Delta_h ) \delta_{\mathbf p} - \Div_h h'_\gamma(D_h {\mathbf u}_k)[D_h \delta_{\mathbf p}] + 2\mathbf p_k \delta_{\boldsymbol{\lambda}} = 2g_2,\notag \\
&  \chi_{I_k}\big\{\mathbf p_k \delta_{\mathbf u} +  (\mathbf {u}_k - \mathbf {f})\, \delta_{\mathbf p}  + \beta(\II - \Delta_h) \delta_{\boldsymbol{\lambda}}\big\}=\chi_{I_k} \beta(\II - \Delta_h) g_3,\label{redu_c}\\
&    \chi_{A_k}\delta_{\boldsymbol{\lambda}} = g_4, \label{redu_d}
\end{align}
\end{subequations}
where $g_1 = F_h^1 - \Div_h F_h^2$, $g_2 = \frac{1}{2}(F_h^3- \Div_h F_h^4)$, $g_3 = \beta^{-1}(\II - \Delta_h)^{-1}F_h^5$ and $g_4=-\chi_{A_k} \boldsymbol{\lambda}_k$.

Next, we show that there exists a neighborhood $V(\mathbf u^*, \boldsymbol{\lambda}^*, \mathbf p^*)$ such that with any $(\mathbf u, \boldsymbol{\lambda}, \mathbf p)\in V(\mathbf u^*, \boldsymbol{\lambda}^*, \mathbf p^*)$ the system (\ref{ssnewton}) is solvable for every right-hand side $F_h^i$.
To show the existence and uniqueness of a solution to (\ref{reduce_system}), let us introduce the following auxiliary problem
\begin{equation}\label{aux_prob}
\begin{aligned}&\min \iJ_A (\delta_{\mathbf u}, \delta_{\boldsymbol{\lambda}}) = \|\delta_{\mathbf u} - g_2\|^2 + \beta\|\chi_{I_k}(\delta_{\boldsymbol{\lambda}} - g_3)\|^2 + \beta\|\chi_{I_k}[D_h(\delta_{\boldsymbol{\lambda}} - g_3)]\|^2\\  &\hspace{3.5cm}+ \dfrac{1}{2}\big\lag e_{\mathbf{uu}}[\delta_{\mathbf u}]^2,\mathbf p_k\big\rag + \big\lag e_{\mathbf u \boldsymbol{\lambda}}[\delta_{\mathbf u}, \delta_{\boldsymbol{\lambda}}],\mathbf p_k\big\rag \\
&\mbox{subject to}
\\ &\qquad e_{\mathbf u}(\mathbf u_k,\boldsymbol{\lambda}_k)\delta_{\mathbf u} + e_{\lambda}(\mathbf u_k,\boldsymbol{\lambda}_k)\delta_{\boldsymbol{\lambda}} = g_1,\\
&\qquad \chi_A \delta_{\boldsymbol{\lambda}} =-\chi_A  \boldsymbol{\lambda}_k.
\end{aligned}\end{equation}
It is not difficult to show that (\ref{reduce_system}) corresponds to the optimality condition for problem (\ref{aux_prob}).
Considering the auxiliary Lagrangian
\begin{equation*}
  \begin{aligned}L (\delta_{\mathbf u}, \delta_{\boldsymbol{\lambda}}, \delta_{\mathbf p}, \psi) &= \iJ_A (\delta_{\mathbf u}, \delta_{\boldsymbol{\lambda}}) + \lag \psi, \chi_A (\delta_{\boldsymbol{\lambda}} -  g_4)\rag+ \lag \delta_{\mathbf p},  e_{\mathbf u}(\mathbf u_k,\boldsymbol{\lambda}_k)\delta_{\mathbf u} + e_{\lambda}(\mathbf u_k,\boldsymbol{\lambda}_k)\delta_{\boldsymbol{\lambda}} - g_1\rag,
\end{aligned}
\end{equation*}
it can be verified that its second derivative is given by
\begin{equation}
  L_{(\delta_{\mathbf u}, \delta_{\boldsymbol{\lambda}})}''[v,l]^2 =2\|v\|^2 + 2\beta(\|l\|^2 +\|D_h l\|^2) + \lag h''_\gamma(D_h {\mathbf u})^*[D_hv]^2, D_h {\mathbf p}\rag + 4\lag \mathbf p v,l\rag.
\end{equation}
By Lemma \ref{characters_h_gamma}, it follows that $e_{\mathbf{uu}}(\mathbf u)$ is Lipschitz continuous. Hence, from (\ref{2nd_order_b}) there exists a neighborhood $V(\mathbf u^*, \boldsymbol{\lambda}^*, \mathbf p^*)$ and a constant $\rho >0$, such that for all
$(\mathbf u, \boldsymbol{\lambda}, \mathbf p) \in V(\mathbf u^*, \boldsymbol{\lambda}^*, \mathbf p^*)$,
\begin{equation} \label{ieq_J_dd_vl}
L_{(\delta_{\mathbf u}, \delta_{\boldsymbol{\lambda}})}''[v,l]^2 \geq \frac{\rho}{2} (\|l\|^2 +\|D_h l\|^2),
\end{equation}
for all $(v,l) \in U_h \times \Lambda_h$ satisfying $e_{\mathbf u}(\mathbf u, \boldsymbol{\lambda})v + e_{\lambda}(\mathbf u, \boldsymbol{\lambda})l = 0$.
Therefore, (\ref{aux_prob}) is a linear quadratic optimization problem with convex objective function,
which implies the solvability of (\ref{reduce_system}).

Multiplying equation \eqref{redu_b} by $\delta_u$ we get that
\begin{equation}
\lag h''_\gamma(D_h {\mathbf u}_k)^* [D_h {\mathbf p}_k,D_h \delta_{\mathbf u}], \delta_{\bu} \rag + 2 \|\delta_{\bu}\|^2 + 2 \lag \bp_k \delta_{\bl}, \delta_{\bu} \rag= -\lag e_{\bu}(\bu,\bl) \delta_{\bp}, \delta_{\bu} \rag+2 \lag g_2, \delta_{\bu} \rag.
\end{equation}
Plugging the last equation in the second order condition \eqref{ieq_J_dd_vl} and using \eqref{eq: uniform invertibility bound}, we get that
\begin{equation} \label{eq: ssn conv ssc}
\frac{\rho}{2} \|(\delta_{\bu}, \delta_{\bl})\|^2 \leq 2 \lag \bp_k \delta_{\bl}, \delta_{\bu} \rag + 2 \beta \chi_{I_k} (\|\delta_{\bl}\|^2 + \|D_h \delta_{\bl}\|^2) -\lag e_{\bu}(\bu,\bl) \delta_{\bp}, \delta_{\bu} \rag+2 \lag g_2, \delta_{\bu} \rag.
\end{equation}
On the other hand, multiplying \eqref{redu_c} by $\delta_{\bl}$ we get that
\begin{equation}
  \chi_{I_k} \left( 2 \beta (\|\delta_{\bl}\|^2 + \|D_h \delta_{\bl}\|^2) +  \lag \bp_k \delta_{\bl}, \delta_{\bu} \rag+ 2\lag (\bu_k- \mathbf f)   \delta_{\bp}, \delta_{\bl} \rag \right) \leq C \|g_3\|_{I_k} \|\delta_{\bl}\|_{I_k}.
\end{equation}
Using the latter in \eqref{eq: ssn conv ssc} we then get that
\begin{align}
  \frac{\rho}{2} \|(\delta_{\bu}, \delta_{\bl})\|^2 &\leq 2 \chi_{A_k}  \lag \bp_k \delta_{\bl}, \delta_{\bu} \rag- 2 \chi_{I_k} \lag (\bu_k- \mathbf f) \delta_{\bp}, \delta_{\bl} \rag\\
  & \hspace{1cm}- \lag e_{\bu}(\bu,\bl)^T \delta_{\bp}, \delta_{\bu} \rag+2 \lag g_2, \delta_{\bu} \rag + C \|g_3\|_{I_k} \|\delta_{\bl}\|_{I_k}\\
  &\leq 2 \|\bp_k\|_{A_k} \| \delta_{\bu}\| \|g_4\|+ 2 \|g_2\| \|\delta_{\bu}\| + C \|g_3\|_{I_k} \|\delta_{\bl}\|_{I_k}\\
  & \hspace{1cm}- \lag \delta_{\bp}, e_{\bu}(\bu,\bl) \delta_{\bu}+ 2(\bu_k- \mathbf f) \delta_{\bl} \rag+ 2 \chi_{A_k} \lag (\bu_k-\mathbf f)\delta_{\bp}, \delta_{\bl}\rag,
\end{align}
where we used the bound $\|\delta_{\bl}\|_{A_k} \leq \|g_4\|$ obtained from equation \eqref{redu_d}. Since $e_{\bl}(\bu,\bl) \delta_{\bl}= 2(\bu_k- \mathbf f) \delta_{\bl}$ we obtain from equation \eqref{redu_a} that
\begin{multline}
  \frac{\rho}{2} \|(\delta_{\bu}, \delta_{\bl})\|^2 \leq 2 \left( \|\bp_k\|_{A_k} \|g_4\|+ \|g_2\| \right) \| \delta_{\bu}\|  \\+ C \|g_3\| \|\delta_{\bl}\|+ \|g_1\| \|\delta_{\bp}\|+ 2 \|\bu_k- \mathbf f\| \|g_4\| \|\delta_{\bp}\|.
\end{multline}
From the uniform invertibility of $e_{\bu}(\bu,\bl)$ and equation \eqref{redu_b} we get that
\begin{equation}
  \|\delta_{\bp}\| \leq K \left( \|\bp_k\| \| \delta_{\bl}\|+ \| \delta_{\bu}\| + \|g_2\| + \|D_h \bu_k\| \|D_h \bp_k\| \|D_h \delta_{\bu}\| \right).
\end{equation}
Using Young's inequality for the term $\|g_2\| \|\delta_{\bu}\|$ we get that
\begin{equation}
\|g_2\| \|\delta_{\bu}\| \leq C\|g_2\|^2 + \frac{\rho}{16} \|(\delta_{\bu},\delta_{\bl})\|^2.
\end{equation}
A similar bound is obtained for the terms $\|g_4\| \|\delta_{\bu}\|$ and $\|g_3\| \|\delta_{\bl}\|$. For the term $\|\delta_{\bp}\| \|g_1\|$ we get that
\begin{equation}
\|\delta_{\bp}\| \|g_1\|  \leq K \|g_1\|^2 + \tilde K \|g_2\|^2 +\frac{\rho}{16} \|(\delta_{\bu},\delta_{\bl})\|^2.
\end{equation}
Altogether we obtain that
\begin{equation}
  \|(\delta_{\bu},\delta_{\bl})\|^2 \leq C \left(\|g_1\|^2 +\|g_2\|^2+\|g_3\|^2+\|g_4\|^2\right),
\end{equation}
which implies the result.
\end{proof}

\subsection{Globalization}
The semismooth Newton method (\ref{ssnewton}) typically exhibits a very small convergence neighbourhood for high values of $\gamma$. In order to globalize the semismooth Newton method, instead of using a line-search strategy, we consider a modified Jacobi matrix in each iteration. The main idea consists in reinforcing feasibility of the dual quantities (with suitable projections) in the building of the Jacobian and, in that manner, obtain a global convergent behaviour of the resulting algorithm.

To describe the modification, let us first introduce the following notation:
\begin{align*}
P_1(\mathbf u) &= \frac{2\gamma-1}{4\gamma} + \frac{\gamma|D_h {\mathbf u}|}{2} - \frac{\gamma}{2} t_1(\mathbf u)t_2(\mathbf u) + \frac{\gamma^3}{2} t_1^2(\mathbf u)t_2^2(\mathbf u),\\ P_2(\mathbf u) &= \frac{\gamma}{2} - \frac{\gamma^2}{2}\big[t_1(\mathbf u) + t_2(\mathbf u)\big] + \gamma^3\big[t_1(\mathbf u) + t_2(\mathbf u)\big]t_1(\mathbf u)  t_2(\mathbf u).
\end{align*}
The proposed building process is based on the properties of the stationary point we look for. Indeed, at the solution $\yy^*$, we know the following:
\begin{itemize}
\item On $\iA_\gamma$: $ {q} = h_\gamma(D_h {\mathbf u}^*)= \frac{D_h {\mathbf u}^*}{|D_h {\mathbf u}^*|}$. On the other hand, $h'_\gamma(D_h {\mathbf u})^*D_h {\mathbf z} = \frac{D_h {\mathbf z}}{|D_h {\mathbf u}|} - \frac{\lag D_h {\mathbf u}, D_h {\mathbf z}\rag}{|D_h {\mathbf u}|^2}\frac{D_h {\mathbf u}}{|D_h {\mathbf u}|}$. Since $\left|\frac{D_h {\mathbf u^*}}{|D_h {\mathbf u^*}|}\right| \leq 1$, by projecting onto the feasible set, we have an approximation of $h'_\gamma(D_h {\mathbf u})D_h$ on $\iA_\gamma$:
    $$(h'_\gamma(D_h {\mathbf u}))^\dag D_h := \frac{D_h}{|D_h {\mathbf u}|}- \frac{\lag D_h {\mathbf u}, D_h\rag}{|D_h {\mathbf u}|^2}\frac{ \mathbf{q}}{\max\{1,|  \mathbf{q}|\} }.$$
\item On $\iS_\gamma$: $ \mathbf{q} = h_\gamma(D_h {\mathbf u^*})= P_1(\mathbf u^*)\frac{D_h {\mathbf u^*}}{|D_h {\mathbf u^*}|}$, $1-\frac{1}{2\gamma}\leq P_1(\mathbf u)\leq 1$ and
    \begin{align*}h'_\gamma(D_h {\mathbf u})^*D_h {\mathbf z} = P_1(\mathbf u)\bigg(\frac{D_h {\mathbf z}}{|D_h {\mathbf u}|} - \frac{\lag D_h {\mathbf u}, D_h {\mathbf z}\rag}{|D_h {\mathbf u}|^3}D_h {\mathbf u} \bigg) +P_2(u)\frac{\lag D_h {\mathbf u}, D_h {\mathbf z}\rag}{|D_h {\mathbf u}|^2} D_h {\mathbf u}\\
    = \bigg(\frac{(D_h {\mathbf z} D \mathbf u^T)}{|D_h {\mathbf u}|^2} - \frac{\lag D_h {\mathbf u}, D_h {\mathbf z}\rag}{|D_h {\mathbf u}|^2} \bigg)P_1(\mathbf u)\frac{D_h {\mathbf u}}{|D_h {\mathbf u}|} +P_2(\mathbf u)\frac{\lag D_h {\mathbf u}, D_h {\mathbf z}\rag}{|D_h {\mathbf u}|^2} D_h {\mathbf u}.
    \end{align*}
    Hence, similar to the above consideration, we obtain:
    $$
    (h'_\gamma(D_h {\mathbf u}))^\dag D_h := \bigg\{\frac{(D_h {\mathbf z} D \mathbf u^T)}{|D_h {\mathbf u}|^2}
    +  \bigg[ \frac{P_2(\mathbf u)}{P_1(\mathbf u)}-\frac{1}{|D_h {\mathbf u}|} \bigg] \frac{\lag D_h {\mathbf u} , D_h\rag}{|D_h {\mathbf u}|}\bigg\}\frac{ \mathbf{q}}{\max\{1, | \mathbf{q}|\} }.
    $$
\end{itemize}
By replacing  $(h'_\gamma(D_h {\mathbf u}))$ by $(h'_\gamma(D_h {\mathbf u}))^\dag$, we get a modified generalized derivative of $F_h$:
\begin{multline}\label{s3.7_j_mod}
\iG^\dag_{F_h}(\yy)(\delta_{\mathbf u}, \delta_{\mathbf q}, \delta_{\mathbf p}, \delta_{\mathbf z}, \delta_{\boldsymbol{\lambda}})^T=\\
\left(\begin{array}{c}
(2 {\boldsymbol{\lambda}} \II -\mu \Delta_h)\delta_{\mathbf u}  -\Div_h \delta_{\mathbf q}   + 2(\mathbf u -\mathbf f) \delta_{\boldsymbol{\lambda}}\\
(h'_\gamma(D_h {\mathbf u}))^\dag\delta_{\mathbf u} -\delta_{\mathbf q} \\
2\delta_{\mathbf u}  + (2 {\boldsymbol{\lambda}} \II -\mu \Delta_h)\delta_{\mathbf p}   -\Div_h\delta_{\mathbf z}    + 2 p \delta_{\boldsymbol{\lambda}} \\
\big(h''_\gamma(D_h {\mathbf u})^*[D_h {\mathbf p}, D_h \delta_{\mathbf u}] + (h'_\gamma(D_h {\mathbf u}))^\dag \delta_{\mathbf p} -\delta_{\mathbf z} \\
p \delta_{\mathbf u} +  (\mathbf  {u} -  \mathbf {f}) \delta_{\mathbf p}  + \beta(\II - \Delta_h) \delta_{\boldsymbol{\lambda}} - G_m\big((\mathbf u -\mathbf f) p- \beta \Delta_h  \boldsymbol{\lambda} \big)\big(p \delta_{\mathbf u} +  (\mathbf{u} - \mathbf{f}) \delta_{\mathbf p}  - \beta\Delta_h \delta_{\boldsymbol{\lambda}}\big)
\end{array} \right)
\end{multline}
and the corresponding modified iteration for solving of $F_h(\yy)=0$ with $F_h$ in (\ref{s3.1_strong_form_h}):
\begin{equation}\label{s3.6_Newton_step_disc}
\iG^\dag_{F_h}({\yy}_k)\big({\yy}_{k+1} - {\yy}_k\big)= - F_h({\yy}_k).
\end{equation}

\section{Computational experiments}
\label{sec:Computational experiments}

All schemes developed previously were implemented in MATLAB and run in a HP Blade multiprocessor system. The overall used algorithm is given through the following steps:
\begin{algorithm}[Domain Decomposition-Semismooth Newton Algorithm]\label{alg_3_1}~\\
\begin{itemize}
\item[0.] Initialize $\yy_0=(\mathbf u_0, \mathbf q_0, \mathbf p_0, \mathbf z_0, \boldsymbol{\lambda}_0)^T$, choose the number of subdomains $M$, the number of intersecting pixels $L$ and set $k=0$.
\item[1.] In each subdomain $j \in \{1, \dots, M\}$, solve iteratively $(\ref{s3.6_Newton_step_disc})$:
    $$\iG^\dag_{F_h}(\yy_k^j) \delta_{\yy}^j = - F_h(\yy_k^j),$$
    until $\| \delta_{\yy}^j \| \leq tol$, and update $\yy_{k+1}^j = \yy_k^j + \delta_{\yy}^j$.
\item[2.] Merge the subdomain solutions $\yy_{k+1}^j$ into one solution $\yy_{k+1}$ on the whole image domain.
\item[3.] Stop if the domain-decomposition stopping criteria is satisfied. Otherwise, update $k \leftarrow k+1$ and go to 1.
\end{itemize}
\end{algorithm}
Since the computations in each subdomain are independent from each other, these may run simultaneously in parallel processors. We implemented a standard for-loop for iteration $k$ of the domain decomposition method and, within each $k$, a parallel MATLAB parfor-loop for computing the solution on each subdomain.

For the numerical experimentation we introduce some notation and several quantities of interest, which are described next:\\
{\small
\begin{tabular}{l l}
$L$ & Number of overlapping pixels between 2 neighboring subdomains\\
 $M_{NonDDC}$ & Semismooth Newton method on the whole domain $\Omega$\\
 $M_{orgDDC}$ & Original Schwarz-Semismooth Newton method\\
 $M_{optDDC}$ & Optimized Schwarz-Semismooth Newton method\\
 $er_\lambda$ & $\|\boldsymbol{\lambda}_{DD}-\boldsymbol{\lambda}\|$, where $\boldsymbol{\lambda}_{DD}$ is obtained by $M_{orgDDC}$ or $M_{optDDC}$, and $\boldsymbol{\lambda}$ by $M_{NonDDC}$\\
 $er_u$ & $\|\mathbf u_{DD}-\mathbf u\|$, where $\mathbf u_{DD}$ is obtained by $M_{orgDDC}$ or $M_{optDDC}$, and $\mathbf u$ by $M_{NonDDC}$.\\
 $k_{\max}$ & Maximum number of subdomain SSN-iterations in all DD iterations\\
 SSNR & $\sum_i\|F^i_h(\yy_{k_{\max}})\|$ on $\Omega_i\subset \Omega$\\
$T_p$ & Performing time (in seconds).
\end{tabular}
}

We also use the structural similarity measure (SSIM) (see \cite{WangBokvik}) to compare the obtained images with the original one.

\subsection{Uniform Gaussian noise}
\noindent In this first experiment, we consider the denoising problem with brain scan images. The first set consists of images of $256\times 256$ pixels and Gaussian noise with zero mean and variance $\sigma=0.0075$. The original and noisy images are shown in Figure \ref{fig_1st_exper}. The domain decomposition-semismooth Newton algorithms run with the parameter values $\gamma = 50$, $\mu = 10^{-13}$, $\beta = 10^{-9}$ and $h=0.01$. The results are shown in Figure \ref{fig_1st_results}. From the surface representation of $\lambda$, we can observe that $\lambda$ is continuous and its shape is related to the one of the original image. In particular, the regularization is stronger in homogeneous regions in the image, and weaker where the image intensity undergoes variations on a smaller scale.
\begin{figure}[h] \centering
  \begin{tabular}{cc}
    \includegraphics[width=0.24\textwidth]{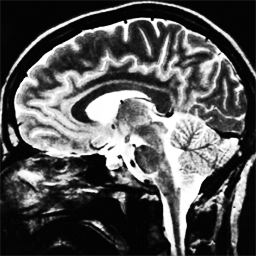} &
    \includegraphics[width=0.24\textwidth]{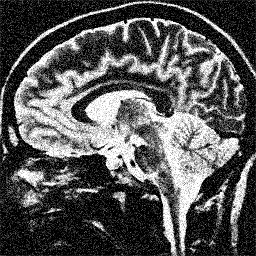}
  \end{tabular}
  \caption{\small The first experiment: Original (left) and noisy (right) images.}
  \label{fig_1st_exper}
\end{figure}
\begin{figure} \centering
  \begin{tabular}{ccc}
    \includegraphics[width=0.235\textwidth]{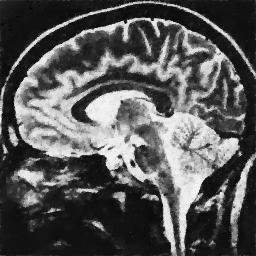} &
    \includegraphics[width=0.46\textwidth]{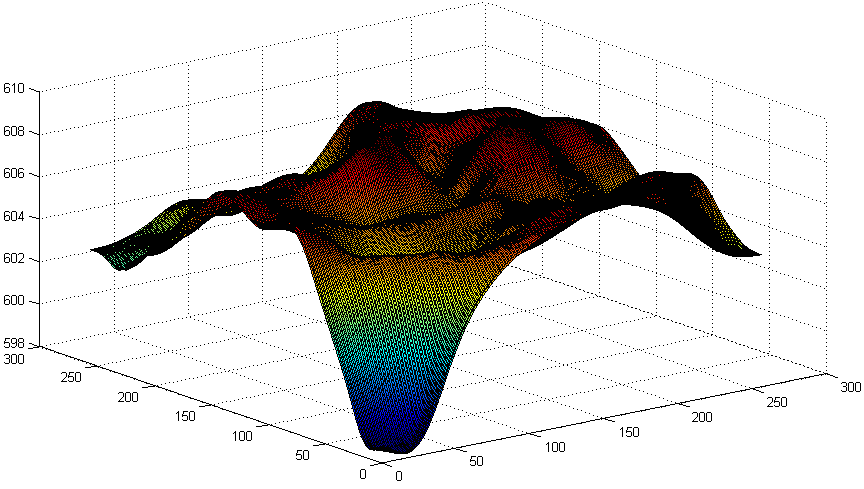} &
    \includegraphics[width=0.235\textwidth]{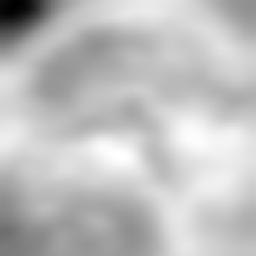}\\
  \includegraphics[width=0.235\textwidth]{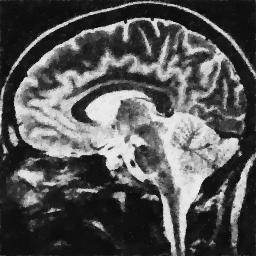}&
    \includegraphics[width=0.46\textwidth]{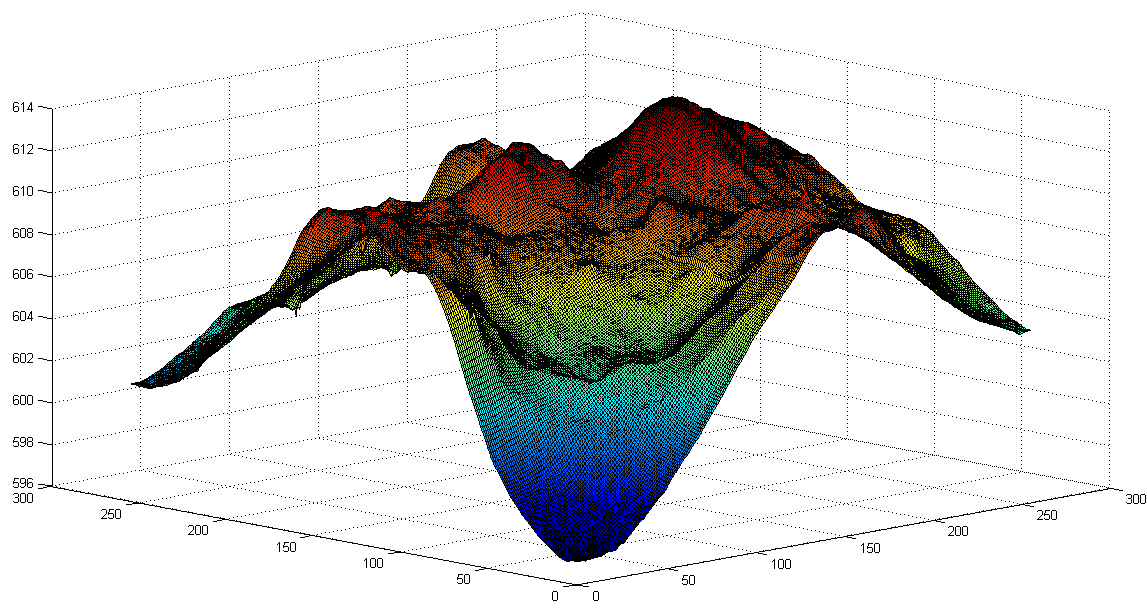}&
    \includegraphics[width=0.235\textwidth]{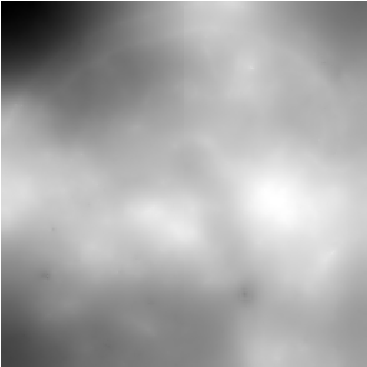}\\
    \includegraphics[width=0.235\textwidth]{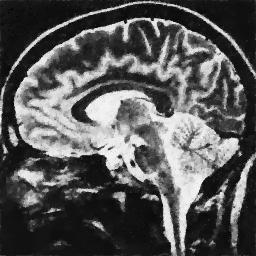}&
      \includegraphics[width=0.46\textwidth]{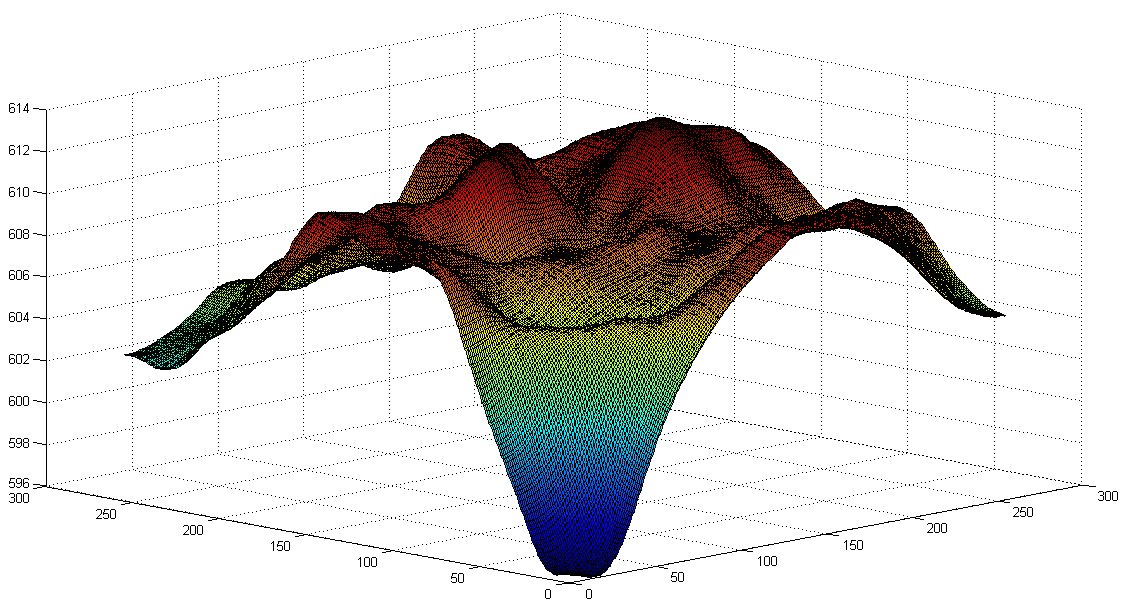}&
    \includegraphics[width=0.235\textwidth]{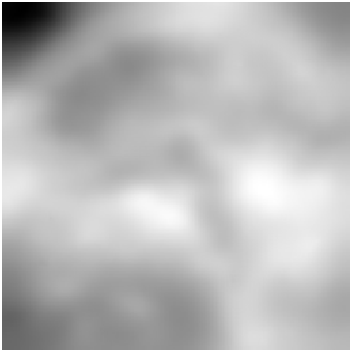}
  \end{tabular}
  \caption{\small Using the training set in Figure \ref{fig_1st_exper} the optimally denoised images are shown (left), surface plots of $\lambda$ (center) and images of $\lambda$ (right). The first row corresponds to the result achieved without domain decomposition $M_{NonDDC}$; the second and third row correspond to the results using domain decomposition (2 iterations) without ($M_{orgDDC}$) and with ($M_{optDDC}$) optimized transmission condition, respectively. Here we used 2 subdomains with an overlap of $L=40$ pixels.}
  \label{fig_1st_results}
\end{figure}

In Table \ref{t3.1_1st_exp} the performance of the different methods is compared. For all of them, only the first 2 domain decomposition iterations were considered. The total number of SSN iterations differ at most by one. The impact of the domain decomposition method becomes clear when comparing the computing times of the methods, corresponding to one, two and four subdomains. The computing time is significantly reduced. 
The effect of the optimized transmission conditions can be realized when comparing the gap between subdomains, which is much lower in the case of  optimized transmission conditions ($M_{optDDC}$) than in the standard Schwarz method ($M_{orgDDC}$).
\begin{table}[h]
{\small
\begin{tabular}{l |  c|c| c|c| c|c |c| c| c|}
\hline
\multirow{2}{*}{Method} & \multirow{2}{*}{$k_{\max}\quad$}  & \multicolumn{4}{c|}{$L=20$}& \multicolumn{4}{c|}{$L=40$}\\
\cline{3-10}
 & &(1)&(2)& (3)& (4) &(1)&(2)& (3) & (4)\\
\hline
 $M_{NonDDC}$ & $10$  & \multicolumn{8}{c|}{$SSIM=0.894\quad T_p= 83.71$} \\
\hline
\multirow{2}{*}{$M_{orgDDC}$} &(a) $11$ &0.851 & 5.3 &$2.71$&$28.11$  &0.861 & 3.1  & $1.76$&$38.01$  \\
\cline{2-10}
&(b) $11$ &0.853 & 5.9 & $3.60$&$10.09$ & 0.858& 3.7 & $2.05$ & $19.99$\\
\hline
\multirow{2}{*}{$M_{optDDC}$} &(a) $11$ &0.869 & 3.2 & $0.99$&$29.85$ &0.881 & 1.9 &$1.01$ & $39.92$\\
\cline{2-10}&(b) $10$ & 0.865 & 3.6 & $1.22$&$11.03$ &0.877  &  2.3 &$1.09$ & $23.81$\\
\hline
\end{tabular} }
\caption{\small Numerical results for the first experiment after one domain decomposition iterations.
 Rows (a): 2 subdomains; (b): 4 subdomains. Columns (1): $SSIM$; (2): $er_{u}$ ($\times 10^{-3}$); (3): $er_{\lambda}$; (4): $T_p$.} \label{t3.1_1st_exp}
\end{table}

\subsection{Non-uniform Gaussian noise}
For this experiment we consider input images of size $512\times512$, with a Gaussian noise of $\sigma=0.014$ on the whole domain and an additional Gaussian noise component of $\sigma=0.016$ on some areas which are marked in red (see Figure \ref{fig_4th_input}). The~parameter~values used are $\mu = 0$, $\beta=10^{-10}$, $\gamma=100$ and $h = 0.002$. The shape of $\lambda$ is shown in Figure \ref{fig_4th_rs_whole}.
\begin{figure}[h] \centering
  \begin{tabular}{cc}
    \includegraphics[width=0.25\textwidth]{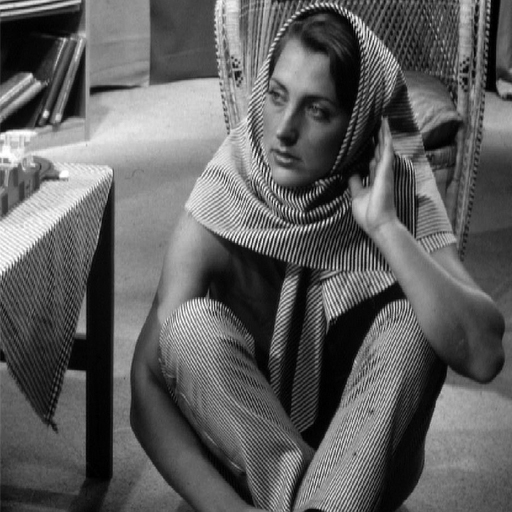} &
    \includegraphics[width=0.25\textwidth]{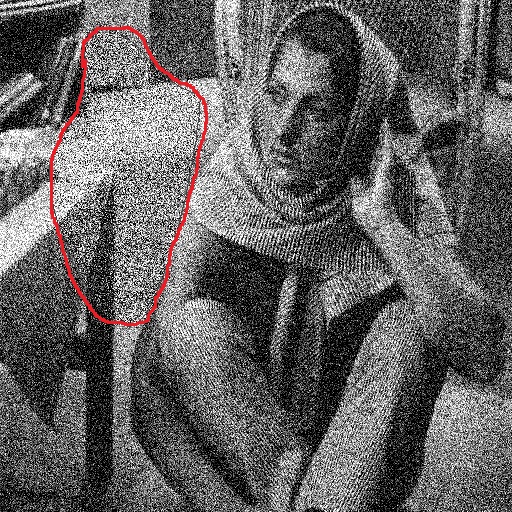}
  \end{tabular}
  \caption{\small The input images for the non-uniform noise experiment: original (left) and noisy (right) images.}
  \label{fig_4th_input}
\end{figure}\enlargethispage*{15pt}

\begin{figure}[h] \centering
  \begin{tabular}{cc}
    \includegraphics[width=0.25\textwidth]{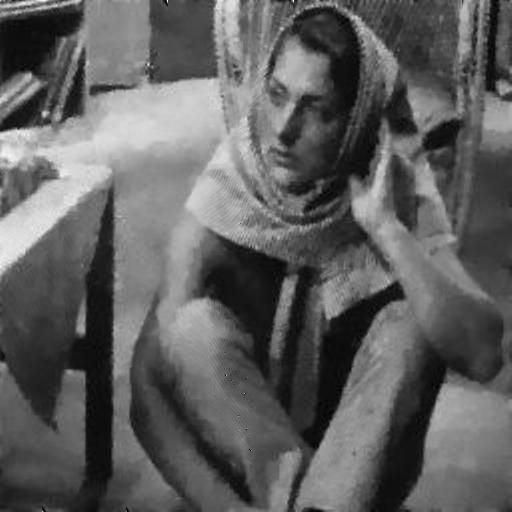} &
    \includegraphics[width=0.25\textwidth]{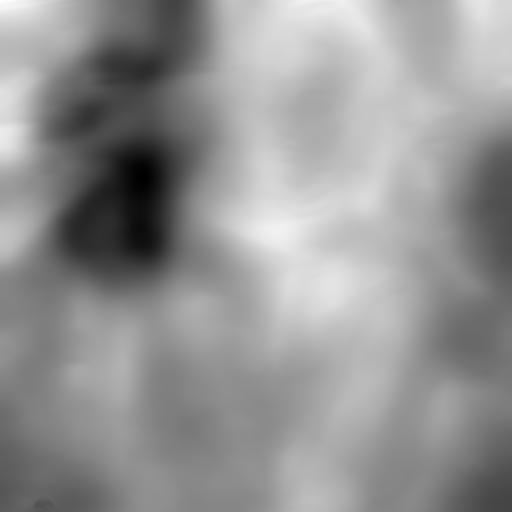}\\
  \end{tabular}
  \caption{\small Denoised image (left) and image of $\lambda$ (right).}
  \label{fig_4th_rs_whole}
\end{figure}

The semismooth Newton method, on the whole domain, takes $k_{\max}=14$ iterations and $T_p = 1398.1(s)$ to converge. The denoised image has an $SSIM = 0.791$. Meanwhile, one iteration of $M_{orgDDC}$ with $L = 30$ takes $k_{\max}=15$ iterations and $T_p = 411.7(s)$ to converge, and yields $SSIM = 0.769$. The error with respect to $\lambda$ is given by $er_{\lambda} = 0.97$.
\begin{figure}[!h] \centering
  \begin{tabular}{cc}
    \includegraphics[width=0.25\textwidth]{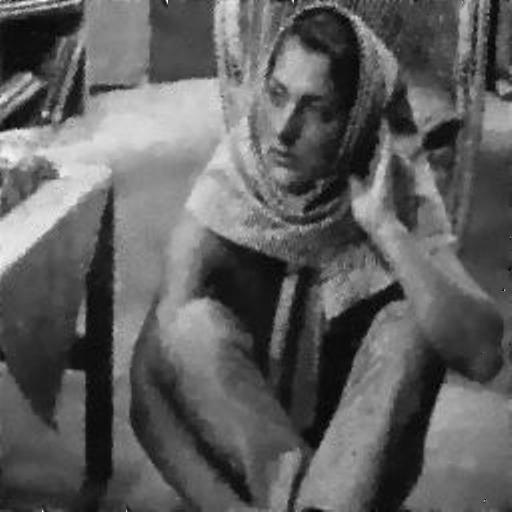} &
    \includegraphics[width=0.25\textwidth]{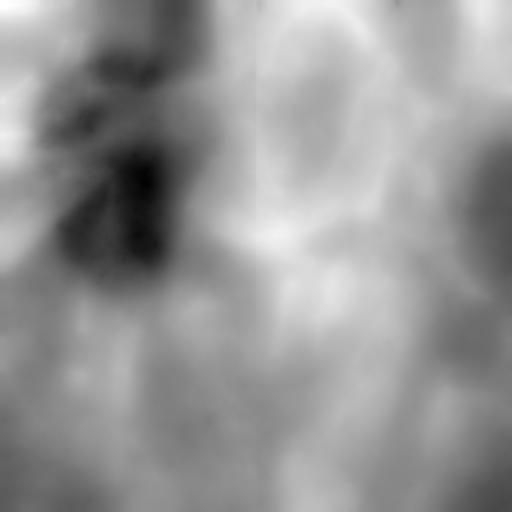}
  \end{tabular}
  \caption{\small $M_{orgDDC}$ with $L = 30$: Denoised image (left) and $\lambda$ (right).}
  \label{fig_4th_rs_org}
\end{figure}
With the same value $L = 30$, the $M_{optDDC}$ stops after $k_{\max}=15$ and $T_p = 433.9(s)$. The similarity measure is $SSIM = 0.785$ and the error with respect to $\lambda$ is given by $er_{\lambda} = 0.51$. The corresponding images for all three methods are given in Figures \ref{fig_4th_rs_whole}, \ref{fig_4th_rs_org} and \ref{fig_4th_rs_opt}, respectively.

\begin{figure}[h] \centering
  \begin{tabular}{cc}
    \includegraphics[width=0.25\textwidth]{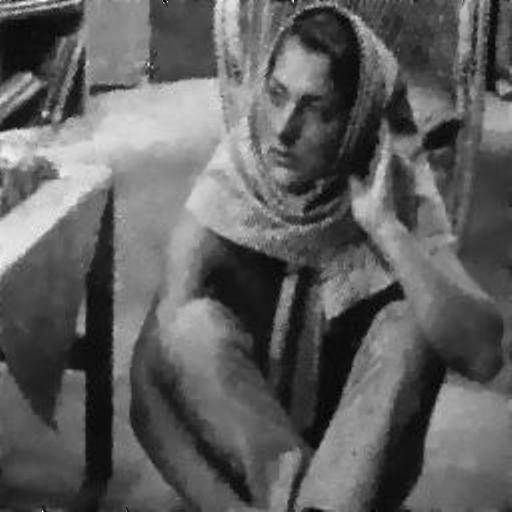} &
    \includegraphics[width=0.25\textwidth]{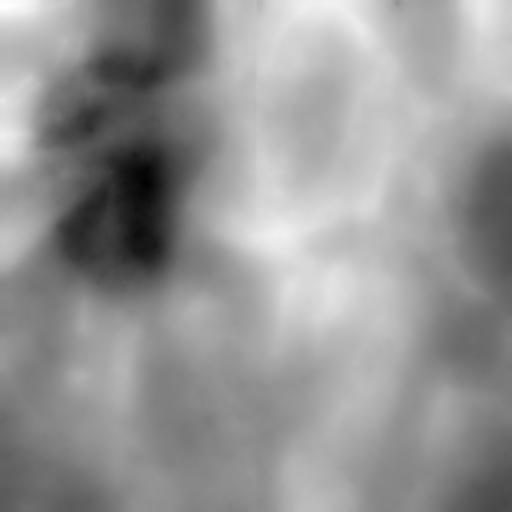}
  \end{tabular}
  \caption{\small $M_{optDDC}$ with $L = 30$: Denoised image (left) and $\lambda$ (right).}
  \label{fig_4th_rs_opt}
\end{figure}

From Figures \ref{fig_4th_rs_whole}, \ref{fig_4th_rs_org} and \ref{fig_4th_rs_opt} we can observe that the areas with higher noise level result in smaller pointwise values of $\lambda$. Moreover, from the tabulated results, one can realize that, in order to get good results for $M_{orgDCC}$, a sufficiently large value of $L$ is required. This has of course an increasing effect in the total computing time.

\subsection{Large training set}
As can be seen in the experiments with one training image, the spatially adapted $\lambda$ does not only capture inhomogeneities in the noise, but also adapts to the scale of structures in the underlying image. Learning one fixed parameter, therefore, for more than one image seems counterintuitive since these local adaptions will change in each image. In the following experiment we argue, however, that if the training set features images with sufficiently similar content as well as with similar and heterogenous noise properties, as might be the case for MRI scans of brains, then the learned, spatially-adapted $\lambda$ still outperforms a learned $\lambda$ that is constant. To verify this, we compute the optimal functional parameter $\lambda$ from a training set of 10 pairs $(u_j^\dag, f_j)$, $j=1, \dots, 10$. The images (of size $256\times 256$) were taken from the OASIS online database. A Gaussian  noise with $\sigma=0.025$ was distributed on the images, and in the areas marked by red, additional noise with $\sigma=0.1$ was imposed (to all noisy images at the same location).

The parameter values used for this experiment were $\gamma = 50$, $\mu = 10^{-15}$, $\beta = 10^{-12}$ and $h=1/256$. We utilized the optimized Schwarz method $M_{optDDC}$, with overlapping size $L=5$, and stop after two iterates. A total amount of 24 subdomains were considered and the computations were carried out in parallel. The semismooth Newton method, within each step of $M_{optDDC}$, stops whenever $err < 0.01$. The results are shown in Figure \ref{fig_training_test}.
\begin{figure}[h!] \centering
  \begin{tabular}{lccccc}
    a)&\includegraphics[width=0.15\textwidth]{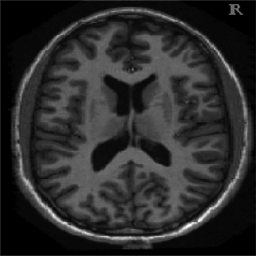} & \includegraphics[width=0.15\textwidth]{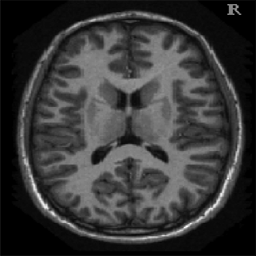} & \includegraphics[width=0.15\textwidth]{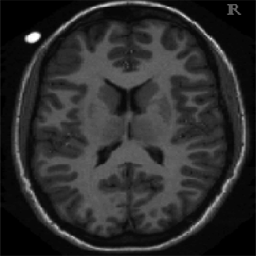} &
    \includegraphics[width=0.15\textwidth]{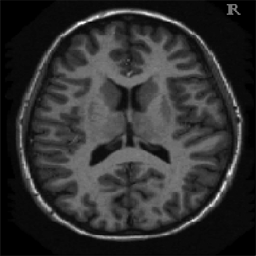} &
    \includegraphics[width=0.15\textwidth]{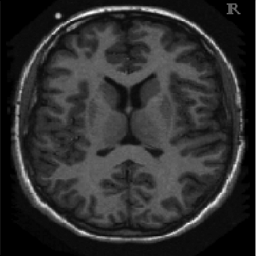} \\
      b)&\includegraphics[width=0.15\textwidth]{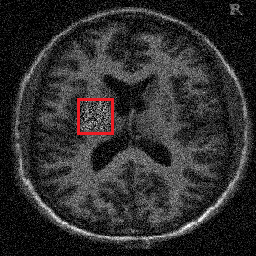}& \includegraphics[width=0.15\textwidth]{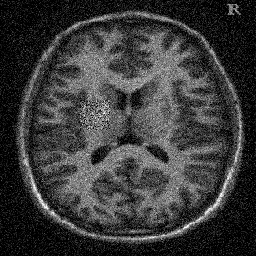}& \includegraphics[width=0.15\textwidth]{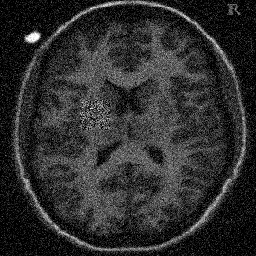}& \includegraphics[width=0.15\textwidth]{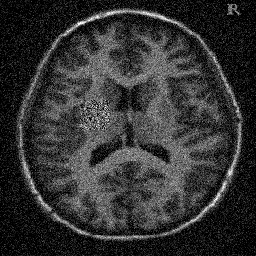}& \includegraphics[width=0.15\textwidth]{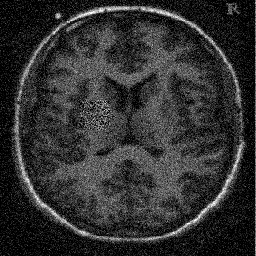}\\
    c)&\includegraphics[width=0.15\textwidth]{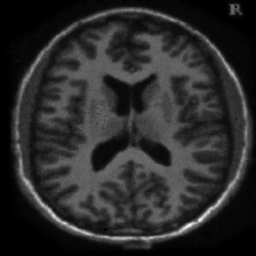}    &
    \includegraphics[width=0.15\textwidth]{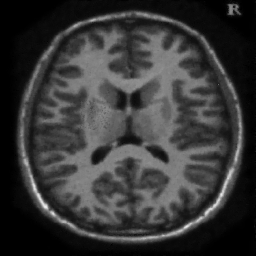}   &
    \includegraphics[width=0.15\textwidth]{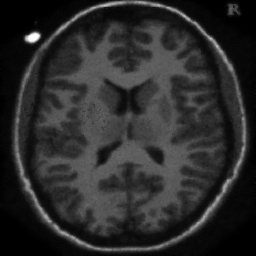}   &
    \includegraphics[width=0.15\textwidth]{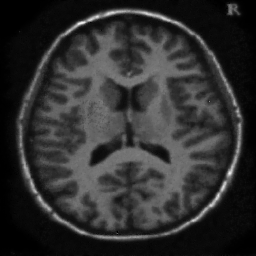}    &
    \includegraphics[width=0.15\textwidth]{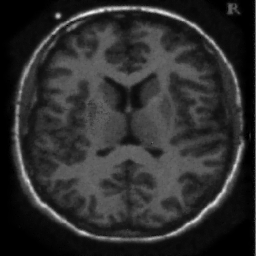}    \\
  \end{tabular}
  \caption{\small Results of learning the spatial parameter $\lambda$ for a training set $(u_k^\dag, f_k)$: (a) Original images, (b) Noisy images, (c) Denoised images with $M_{optDDC}$ (24 subdomains).}
  \label{fig_training_test}
\end{figure}

\begin{figure}[h!] \centering
  \begin{tabular}{cc}
\includegraphics[width=0.7\textwidth]{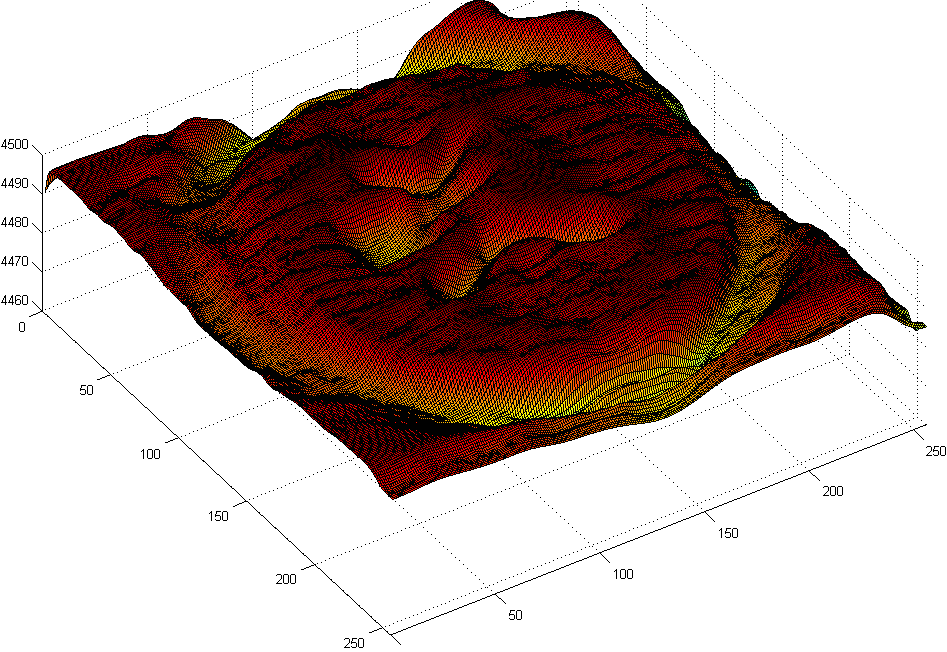}
  \end{tabular}
  \caption{\small Optimal parameter $\lambda$ for the experiment in Figure \ref{fig_training_test} after $2$ Schwarz iterations.}
\end{figure}

The performance of the overall algorithm for the cases of 4 and 24 subdomains is registered in Table \ref{tab_training_test}. It becomes clear from the data, that there is a significant decrease in the total computing time, when an increasing number of subdomains is considered. This, on the other hand, does not significantly affect the quality of the obtained image, measured by SSIM. We denote $AVG_{Gap_{\lambda}} := \frac{1}{10}\sum\limits_{i=j}^{10}\|\lambda_j^m - \lambda_j^n\|\big|_{\Omega_m \cap \Omega_n}$, $j=1, \dots, 10$, $\lambda_j^l = \lambda_j\big|_{\Omega_l}$ and $\Omega_m, \Omega_n$ are subdomains.
\begin{table}[H]
{\small
\begin{tabular}{|c | c| c| c| c | c | c |}
\hline
{$\# \Omega_i$} & {$k_{\max}\quad$}  & $T_p$ & $SSIM_{\min}$ &  ${SSIM}_{\max}$& ${SSIM}_{avg}$ & $AVG_{Gap_{\lambda}}$\\
\hline
4 &$17$ &$2098.42$ & $0.826$  & $0.878$& $0.856$& $3.072$\\
\hline
24 & $14$ &$179.01$ & $0.821$ & $0.883$&$0.863$& $2.785$\\
\hline
\end{tabular} }
\caption{\small Numerical results for $M_{optDDC}$. $SSIM_{\min}$, ${SSIM}_{\max}$, ${SSIM}_{avg}$: min, max and average SSIM of the optimal subdomain images with respect to $u_j^\dag$, $j=1\ldots 10$.}\label{tab_training_test}
\end{table}

\subsection{Performance compared to other spatially-dependent approaches}
In the last experiment we compare the results of our optimal learning approach with the ones obtained with the spatially adapted total variation method (SA-TV) proposed in \cite{dong2011automated}. For the comparison, we apply the optimal spatially-dependent parameter computed in the previous experiment (see Figure 4.8) to a different brain scan, not included in the training set.


The chosen parameters for SA-TV are ${\mu} = 1e-6$, ${\beta} = 10^{-3}$, $\lambda_0=2,5$, $w=11$ and $z=2$. We use the stopping rule as in \cite{dong2011automated}, i.e., $\|u_k - f\|\leq \sigma$. We should remark that the obtained results are very sensitive with respect to the choice of the algorithmic parameters. A lot of trial and error has to be carried out to get proper parameters. This time-consuming preprocessing step should also been taken into account when judging the overall SA-TV performance.
\begin{figure}[h!] \centering
  \begin{tabular}{cccc}
    \includegraphics[width=0.25\textwidth]{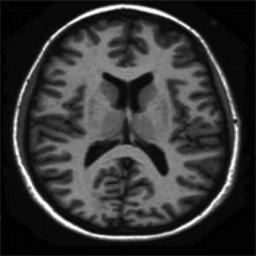} &
    \includegraphics[width=0.25\textwidth]{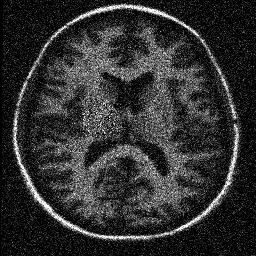}\\
    (a) & (b)\\
    \includegraphics[width=0.25\textwidth]{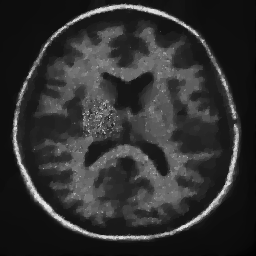}&
    \includegraphics[width=0.25\textwidth]{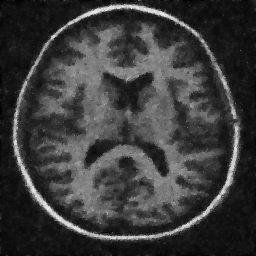}
    \\
    (c)&(d)
  \end{tabular}
  \caption{\small Comparison of SA-TV and the optimal learning approach: (a) Original image, (b) Noisy image, (c) Denoised image with SA-TV, (d) Denoised image with learned $\lambda$}
  \label{fig_last_exper_1}
\end{figure}

We compare our optimal learning method with SA-TV by means of two well-known quality measures: the peak signal-to-noise ratio (PSNR) and the structural similarity measure (SSIM). The results of the two approaches are reported in Table \ref{tab_last_expr}, where it can be observed that our approach outperforms SA-TV for the tested image, with respect to both quality measures.


\begin{table}[h!]
\begin{tabular}{|c| r| c |}
\hline
Method  &  PSNR & SSIM\\
\hline
SA-TV  & 25.31 & 0.799\\
\hline
Learning  & 27.51 &0.822\\
\hline
\end{tabular}
\caption{\small Comparison of our optimal learning approach and SA-TV for the brain scan image with non-uniform gaussian noise} \label{tab_last_expr}
\end{table}

\section{Appendix}
\label{sec:Appendix}

\begin{proof}[Proof of Lemma $\ref{characters_h_gamma}$]
For $z, \xi, \tau \in \RR^2$, by setting $t_1^z=\frac{\gamma}{2}\big(\gamma |z| - 1 + \frac{1}{2\gamma}\big)$, $t_2^z= \gamma |z| -1 - \frac{3}{2\gamma}$,
$$\chi_{\iA_z}=\begin{cases}1 \quad \mbox{if}\quad \gamma|z| \geq b \\ 0 \quad \mbox{otherwise}\end{cases};\chi_{\iS_z}=\begin{cases}1 \quad \mbox{if}\quad a\leq \gamma|z| < b \\ 0 \quad \mbox{otherwise}\end{cases}\mbox{and}\quad \chi_{\iI_z}=\begin{cases}1 \quad \mbox{if}\quad  \gamma|z| < a \\ 0 \quad \mbox{otherwise}\end{cases}$$
\begin{align*}
\mbox{we get}\qquad h'_\gamma(z)[\xi] &=
\chi_{\iA_z}\bigg[\dfrac{\xi}{|z|} - \dfrac{\lag z, \xi\rag}{|z|^3}z\bigg] + \chi_{\iS_z}\bigg\{
\frac{\gamma}{2} \xi + \gamma^2 \big(\gamma|z| -1\big)\big( 2\gamma^2 t_1^zt_2^z -1\big) \frac{\lag z, \xi\rag}{|z|^2} z \qquad\\
&+ \bigg[\frac{2\gamma-1}{4\gamma} - \frac{\gamma t_1^zt_2^z}{2} +\frac{\gamma^3(t_1^zt_2^z)^2}{2}\bigg]\bigg( \frac{\xi}{|z|} - \frac{\lag z, \xi\rag}{|z|^3} z\bigg)\bigg\}+ \chi_{\iI_z}\big( \gamma \xi \big).
\end{align*}
Moreover, by setting $\phi(z, \xi) =-\bigg\{\frac{(\xi z^T)}{|z|^3} + \frac{(z \xi^T)}{|z|^3} - 3\frac{\lag z, \xi\rag (z z^T)}{|z|^5}+ \frac{\lag z, \xi\rag}{|z|^3}\bigg\} $, we get
\begin{align*}h''_\gamma(z)[\xi,\tau]  =
\chi_{\iA_z}\phi(z, \xi) \tau + \chi_{\iS_z}&\bigg\{
\phi(z, \xi) \tau \bigg[\frac{\gamma}{2}t_1^zt_2^z\bigg( 4\gamma^3 |z| \big(\gamma|z| - 1\big)- \gamma^2t_1^zt_2^z + 1 \bigg)  \\
- &\bigg(\gamma^3 |z|^2 - \gamma^2|z| +\frac{1}{2} - \frac{1}{4\gamma}\bigg)\bigg] +6\gamma^5  t_1^zt_2^z\dfrac{\lag z, \xi\rag (z z^T)}{|z|^3}\tau \bigg\}.
\end{align*}
$a$) We first consider the case $z, \hat{z}, \xi, \tau \in \RR^2$. Indeed,\\
{($a1$)} If $|z|<\frac{a}{\gamma}$ and $|\hat{z}| < \frac{a}{\gamma}$, we have $|R(z, \hat{z}, \xi, \tau)|:=|h''_\gamma(z)[\xi][\tau] - h''_\gamma(\hat{z})[\xi][\tau] |=0$.\\
{($a2$)} If $|z|>\frac{b}{\gamma}$ and $|\hat{z}| < \frac{a}{\gamma}$, by a straight computation, we find $|z - \hat{z}| \geq \big||z| - |\hat{z}|\big|\geq \frac{1}{2\gamma^2}$ 
and $\big|R(z, \hat{z}, \xi, \tau)\big| = \big|\phi(z, \xi)\tau\big|\leq  \frac{24\gamma^4}{(2\gamma+1)^2}|\xi||\tau|$. This yields (\ref{estima_dd_h}).\\
{($a3$)} If $|z|, |\hat{z}| >\frac{b}{\gamma}$, we have $\frac{1}{|\hat{z}|^3}, \frac{1}{ |z|^3} \leq \big(\frac{1}{\gamma} + \frac{1}{2\gamma^2}\big)^{-3}$ and 
\begin{align*}R(z, \hat{z}, \xi, \tau) =
\bigg\{3\bigg[\frac{\lag z, \xi\rag\big(z z^T\big)}{|z|^5}&- \frac{\lag \hat{z}, \xi\rag\big(\hat{z} \hat{z}^T\big)}{|\hat{z}|^5}\bigg]
- \bigg[\frac{\big(\xi z^T\big)}{|z|^3} - \frac{\big(\xi \hat{z}^T\big)}{|\hat{z}|^3} \bigg]- \bigg[\frac{\big(z \xi^T\big)}{|z|^3} - \frac{\big(\hat{z}\xi^T\big)}{|\hat{z}|^3}\bigg]\\
&- \bigg[\frac{\lag z, \xi\rag}{|z|^3} - \frac{\lag \hat{z}, \xi\rag}{|\hat{z}|^3} \bigg]\bigg\}\tau
=:(3S_0  - S_1  -  S_2 -  S_3 ) \tau.
\end{align*}
One gets $
|S_1| = \big|\frac{\lag z, \xi\rag}{|z|^3} - \frac{\lag \hat{z}, \xi\rag}{|\hat{z}|^3}\big|
\leq \big[\frac{1}{|\hat{z}|^3}+\frac{1}{ |z|^3}\big] \big|\lag z- \hat{z}, \xi\rag \big| + \frac{\big||z|^3\lag z, \xi\rag - |\hat{z}|^3 \lag \hat{z}, \xi\rag\big|}{|z|^3|\hat{z}|^3}
$. We find for the first term $\big[\frac{1}{|\hat{z}|^3}+\frac{1}{ |z|^3}\big] \big|\lag z-\hat{z}, \xi\rag \big| \leq  \frac{16\gamma^6}{(2\gamma+1)^3}  | z-\hat{z}|| \xi|$ and for the second
\begin{align*}
\frac{\big||z|^3\lag z, \xi\rag - |\hat{z}|^3 \lag \hat{z}, \xi\rag\big|}{|z|^3|\hat{z}|^3}
& \leq \frac{|\xi|}{|z|^3|\hat{z}|^3} \big| |z|^3 z - |\hat{z}|^3  \hat{z}\big|\\
& =\frac{|\xi|}{|z|^3|\hat{z}|^3} \big||\hat{z}|^3 (z - \hat{z}) + z\big[|z|^3 -|\hat{z}|^3 \big]\big|\\
& \leq |\xi|.|z - \hat{z}|\big[\frac{1}{|z|^3} + \frac{1}{|\hat{z}|^3} + \frac{1}{|z|.|\hat{z}|^2} + \frac{1}{|z|^2|\hat{z}|}\big]\leq\frac{32\gamma^6}{(2\gamma+1)^3}  | z - \hat{z}|| \xi|.
\end{align*}
Hence, $| S_1 \tau| \leq \frac{48\gamma^6}{(2\gamma+1)^3}  | z - \hat{z}|| \xi||\tau|$.\\
We also have
$|S_2 \tau |
 = \frac{\big|(|\hat{z}|^3z - |z|^3\hat{z}) \lag \xi, \tau\rag\big|}{|z|^3|\hat{z}|^3}\leq \big[\frac{1}{|\hat{z}|^3} +  \frac{1}{|z|^3}\big]\big|z - \hat{z}\big||\xi|| \tau| + \frac{\big||z|^3 z- |\hat{z}|^3\hat{z} \big|}{|z|^3|\hat{z}|^3}|\xi|| \tau|$
and
$|S_3 \tau |  = \big|\frac{\xi\lag z, \tau \rag}{|z|^3} - \frac{\xi\lag \hat{z}, \tau\rag}{|\hat{z}|^3} \big|=\big| \xi \big\lag \frac{|\hat{z}|^3z - |z|^3\hat{z}}{|z|^3|\hat{z}|^3}, \tau  \big\rag \big|\leq |\xi| \big| \frac{|\hat{z}|^3z - |z|^3\hat{z}}{|z|^3|\hat{z}|^3}\big| |\tau|$.\\
Similarly, we have $| S_2 \tau|\leq \frac{48\gamma^6}{(2\gamma+1)^3}  | z - \hat{z}|| \xi||\tau|, |S_3 \tau|\leq \frac{32\gamma^6}{(2\gamma+1)^3}  | z - \hat{z}|| \xi||\tau| $.\\
We get $|S_0\tau| \leq |\xi| |\tau| \big[\big|\frac{|\hat{z}|^3 z -  |z|^3 \hat{z}}{|z|^3|\hat{z}|^3}\big| +  \big|\frac{(z_1z_2)z}{|z|^5} - \frac{(\hat{z}_1\hat{z}_2)\hat{z}}{|\hat{z}|^5}\big|\big]$, where $z=(z_1,z_2), \hat{z}=(\hat{z}_1,\hat{z}_2)$.
Similar to $S_3\tau$, we have $\big|\frac{|\hat{z}|^3z-  |z|^3 \hat{z}}{|z|^3|\hat{z}|^3}\big| \leq \frac{32\gamma^6}{(2\gamma+1)^3}  | z - \hat{z}|$. By setting $\bar{z} = (\bar{z}_1,\bar{z}_2)=\frac{z}{|z|}$ and  $\mathbf{z} =(\mathbf{z}_1,\mathbf{z}_2)= \frac{\hat{z}}{|\hat{z}|}$ one gets $\big|\frac{(z_1z_2)z}{|z|^5} - \frac{(\hat{z}_1\hat{z}_2)\hat{z}}{|\hat{z}|^5}\big|
\leq \big[\frac{1}{|\hat{z}|^3} + \frac{1}{|z|^3}\big]|z - \hat{z}| +\big| \frac{ |z|^3(\mathbf{z}_1\mathbf{z}_2)z - |\hat{z}|^3 (\bar{z}_1\bar{z}_2)\hat{z}}{|z|^3|\hat{z}|^3}\big|$.
We~find
$\big| \frac{ |z|^3(\mathbf{z}_1\mathbf{z}_2)z - |\hat{z}|^3 (\bar{z}_1\bar{z}_2)\hat{z}}{|z|^3|\hat{z}|^3}\big|
   \leq \frac{ |(\mathbf{z}_1\mathbf{z}_2)z -  (\bar{z}_1\bar{z}_2)\hat{z}|}{|\hat{z}|^3} + |z - \hat{z}|\big[\frac{1}{|z||\hat{z}|^2} + \frac{1}{|z|^2|\hat{z}|} + \frac{1}{|z|^3}\big].
$
Without loss of generality, we assume that $|z| \leq |\hat{z}|$. One can verify that
$\big| (\mathbf{z}_1\mathbf{z}_2)z - (\bar{z}_1\bar{z}_2)\hat{z}\big|
\leq \frac{|z -\hat{z}|}{2} + \frac{|\hat{z}|}{2} |\mathbf{z} - \bar{z}||\mathbf{z} + \bar{z}|$ and $|\mathbf{z} - \bar{z}| \leq
\frac{2|\hat{z}-z|}{|z|}$.
 It follows $\frac{\big| (\mathbf{z}_1\mathbf{z}_2)z - (\bar{z}_1\bar{z}_2)\hat{z}\big|}{|\hat{z}|^3}\leq \frac{5|z-\hat{z}|}{2|\hat{z}|^3}$. Hence, we have
$
|S_0\tau| \leq |\xi||\tau||z - \hat{z}|\big\{ \frac{32\gamma^6}{(2\gamma+1)^3}
+ \frac{2}{|z|^3}+ \frac{7}{2|\hat{z}|^3} +  \frac{1}{|z||\hat{z}|^2}+ \frac{1}{|z|^2|\hat{z}|}  \big\}
\leq \frac{96\gamma^6}{(2\gamma+1)^3}|z - \hat{z}||\xi||\tau|
$
and therefore, $|R(z, \hat{z}, \xi, \tau)| \leq \frac{220\gamma^6}{(2\gamma+1)^3}|z - \hat{z}||\xi||\tau|$.\\
{($a4$)} If $a \leq \gamma|z|, \gamma|\hat{z}| \leq b$ then $0 \leq t^z_1, t_1^{\hat{z}} \leq \frac{1}{\gamma}$; $-\frac{1}{\gamma} \leq t_2^z , t_2^{\hat{z}} \leq 0$ and $|\phi(z, \xi)|, |\phi(\hat{z}, \xi)|\leq \frac{24\gamma^4|\xi|}{(2\gamma-1)^2}$. By setting $q(z) = \frac{\gamma}{2}t_1^zt_2^z\big[ 4\gamma^3 |z| \big(\gamma|z| - 1\big)- \gamma^2t_1^zt_2^z + 1 \big]
- \big[\gamma^3 |z|^2 - \gamma^2|z| +\frac{1}{2} - \frac{1}{4\gamma}\big]$ we have
$$
R(z, \hat{z}, \xi, \tau) = \bigg\{\big[q(z)\phi(z, \xi) -q(\hat{z})\phi(\hat{z}, \xi)\big] + 6\gamma^5  \bigg[\frac{t_1^zt_2^z\lag z, \xi\rag (z z^T)}{|z|^3} - \frac{t_1^{\hat{z}}t_2^{\hat{z}}\lag \hat{z}, \xi\rag (\hat{z} \hat{z}^T)}{|\hat{z}|^3}\bigg]\bigg\}\tau
$$
and $|q(z)|, |q(\hat{z})| \leq \gamma(1+\frac{1}{2\gamma})( 2+\frac{1}{2\gamma}) + \frac{6\gamma + 5}{4\gamma}$. We now analyze each term.
\begin{align*}
\big| \big[q(z)\phi(z, \xi) -q(\hat{z})\phi(\hat{z}, \xi)\big]\tau\big| \leq \big| q(z)-q(\hat{z})\big|\big|\phi(z, \xi) \big||\tau|
+ |q(\hat{z})|\big|\phi(z, \xi) -\phi(\hat{z}, \xi)\big| |\tau|.
\end{align*}
 Similarly for (a3), we get $\big|\big[\phi(z, \xi) -\phi(\hat{z}, \xi)\big]\tau\big|\leq \frac{220\gamma^6}{(2\gamma-1)^3} |z-\hat{z} ||\xi| |\tau|$. Besides,
\begin{align*}
\big| q(z)-q(\hat{z})\big| &\leq \frac{\gamma}{2}\big|t_1^zt_2^z-t_1^{\hat{z}}t_2^{\hat{z}}\big|\bigg| 4\gamma^3 |z| \big(\gamma|z| - 1\big)- \gamma^2t_1^zt_2^z + 1 \bigg|  \\
+ \frac{\gamma}{2}\big|t_1^{\hat{z}}t_2^{\hat{z}}\big|\bigg[ 4\gamma^4 \big||z|^2  - |\hat{z}|^2\big| &+ \gamma^3 \big||z| - |\hat{z}|\big| +  \gamma^2\big|t_1^zt_2^z-t_1^{\hat{z}}t_2^{\hat{z}}\big| \bigg]
+\gamma^3 \big||z|^2 -|\hat{z}|^2\big|+ \gamma^2\big||z|- |\hat{z}|\big|.
\end{align*}
From $t_1^zt_2^z = \gamma^2|z|^2 - (a+b) |z| + ab$, it follows
$\big|t_1^zt_2^z-t_1^{\hat{z}}t_2^{\hat{z}}\big| \leq \gamma^2\big| |z|^2 -|\hat{z}|^2 \big| + |a+b|\big||z|- |\hat{z}|\big|$. Note that $\big| |z|^2 -|\hat{z}|^2 \big| =\big|(|z|- |\hat{z}|)(|z|+ |\hat{z}|)\big| \leq \frac{2\gamma + 1}{\gamma^2}|z- \hat{z}|$.
Hence, there exists constant $m_1(\gamma)>0$ only dependent on $\gamma$, such that $\big| \big[q(z)\phi(z, \xi)-q(\hat{z})\phi(\hat{z}, \xi)\big]\tau\big| \leq m_1(\gamma) |z-\hat{z}||\xi| |\tau|$.\\
For the second term $\big|\frac{t_1^zt_2^z\lag z, \xi\rag (z z^T)}{|z|^3} - \frac{t_1^{\hat{z}}t_2^{\hat{z}}\lag \hat{z}, \xi\rag (\hat{z} \hat{z}^T)}{|\hat{z}|^3} \big|=:T_2(z, \hat{z}, \xi) $, we have
$$T_2(z, \hat{z}, \xi) \leq \frac{ |t_1^zt_2^z - t_1^{\hat{z}}t_2^{\hat{z}}|\big|\lag z, \xi\rag (z z^T)\big|}{|z|^3} + |t_1^{\hat{z}}t_2^{\hat{z}}|\bigg|\frac{\lag z, \xi\rag (z z^T)}{|z|^3} - \frac{\lag \hat{z}, \xi\rag (\hat{z} \hat{z}^T)}{|\hat{z}|^3}\bigg|.$$
We get again the expressions as in the first term and case ($a3$). Hence, there exists a constant $m_2(\gamma) > 0$ only depending in $\gamma$, such that $|R(z, \hat{z}, \xi, \tau)|\leq m_2(\gamma) |z-\hat{z}||\xi| |\tau|$.\\
($a5$) If $a \leq \gamma|z| \leq b$ and $\gamma|\hat{z}| < a$ then $h''(\hat{z})[\xi][\tau] = 0$ and hence $|R(z, \hat{z}, \xi, \tau)| = |h''(z)[\xi][\tau]|$.
Similarly to cases (a3) and (a4), we have $|\phi(z, \xi)||\tau|\leq \frac{24\gamma^4|\xi||\tau|}{(2\gamma-1)^2}$ and $\frac{\big|\lag z, \xi\rag (z z^T)\tau\big|}{|z|^3}  \leq  |\xi||\tau|$.
From  $|t_1^z|, |t_2^z|\leq \frac{1}{\gamma}$ it follows that $\frac{\gamma}{2}|t_1^zt_2^z|\big| 4\gamma^3 |z| \big(\gamma|z| - 1\big)- \gamma^2t_1^zt_2^z + 1 \big| \leq (\gamma+\frac{3}{2})|t_1^z|$ and $
6\gamma^5  |t_1^zt_2^z|\bigg|\dfrac{\lag z, \xi\rag (z z^T)}{|z|^3}\bigg| \leq 6\gamma^4|t_1^z||\xi|$.\\
Note that $0 \leq \gamma|\hat{z}| \leq a \leq \gamma|z|$, hence $0 \leq t_1^z = \gamma|z| - a \leq \gamma|z| - \gamma|\hat{z}|$
and therefore $|t_1^z| \leq \gamma(|z| - |\hat{z}|) \leq \gamma |z-\hat{z}|$. Besides, $\big|\gamma^3 |z|^2 - \gamma^2|z| +\frac{1}{2} - \frac{1}{4\gamma}\big| = \gamma\big|(\gamma|z| - \frac{1}{2\gamma})(\gamma|z| - 1 + \frac{1}{2\gamma})\big|
= \gamma\big|\gamma|z| - \frac{1}{2\gamma}\big| |t_1^z| \leq \gamma^2|z-\hat{z}|$.
Hence there exists constant $m_3 (\gamma) > 0$ only dependent on $\gamma$ such that $|R(z, \hat{z}, \xi, \tau)|\leq m_3(\gamma) |z-\hat{z}||\xi| |\tau|$.\\
($a6$) If $a \leq \gamma|\hat{z}| \leq b$ and $\gamma|z| > b$ then
\begin{align*}
&R(z, \hat{z}, \xi, \tau) = \big[\phi(z, \xi)-\phi(\hat{z}, \xi)\big]\tau
 +\bigg\{6\gamma^5  t_1^zt_2^z\dfrac{\lag z, \xi\rag (z z^T)}{|z|^3}\\
 &+\frac{\gamma}{2}t_1^zt_2^z\big[ 4\gamma^3 |z| \big(\gamma|z| - 1\big)- \gamma^2t_1^zt_2^z + 1 \big] \phi(z, \xi)
+ \bigg[\gamma^3 |z|^2 - \gamma^2|z| -\frac{1}{2} - \frac{1}{4\gamma}\bigg]\phi(z, \xi)   \bigg\}\tau.
\end{align*}
We proceed as in case ($a4$) and get $\big|\phi(z, \xi)-\phi(\hat{z}, \xi)\big||\tau|\leq m_4(\gamma)|z-\hat{z}||\xi| |\tau|$ for some constant $m_4(\gamma) >0$. For the remaining terms, from $\gamma|\hat{z}| \geq b \geq \gamma |z| \geq a$ it follows $0 \leq |t_2^z| = |\gamma|z| - b| = b - \gamma |z| \leq \gamma|\hat{z}| - \gamma |z| \leq \gamma|\hat{z} -z|$. Besides,
$\gamma^3 |z|^2 - \gamma^2|z| -\frac{1}{2} - \frac{1}{4\gamma} = \gamma \big[\gamma|z| +\frac{1}{2\gamma}\big]\big[\gamma|z| -1 -\frac{1}{2\gamma}\big] = \gamma \big[\gamma|z| +\frac{1}{2\gamma}\big]t_2^z$.
We process similarly in case ($a5$) and have $$|R(z, \hat{z}, \xi, \tau)| \leq m_4(\gamma)|z-\hat{z}||\xi| |\tau| + m_5(\gamma) |t^z_2| |\xi| |\tau|\leq m_6(\gamma)|z-\hat{z}||\xi| |\tau|$$
where $m_4(\gamma), m_5(\gamma), m_6(\gamma)$ are positive constants only dependent on $\gamma$.\\
All other cases can be deduced from the previous ones, by an exchanging the roles of $z$ and $\hat{z}$. It~is~easy~to see that the above result also holds in case $z, \hat{z}, \xi, \tau \in \RR^N\times \RR^N$ ($N \in \IN^*$).
\medskip

\end{proof}

\bibliographystyle{plain}
\bibliography{ReportChungJC_rp}
\end{document}